\documentclass[a4paper,reqno]{amsart}

\usepackage[english]{babel}

\usepackage{amsmath,amsthm,amssymb,amsfonts}

\usepackage[normalem]{ulem}

\usepackage{mathrsfs}

\usepackage{enumitem}
\setlist[enumerate]{leftmargin=*}

\usepackage{hyperref}

\theoremstyle{plain}
\newtheorem{theorem}{Theorem}[section]
\newtheorem*{theorem*}{Theorem}
\newtheorem{lemma}[theorem]{Lemma}
\newtheorem*{lemma*}{Lemma}
\newtheorem{corollary}[theorem]{Corollary}
\newtheorem{proposition}[theorem]{Proposition}

\theoremstyle{remark}
\newtheorem{remark}[theorem]{Remark}
\newtheorem*{remark*}{Remark}

\theoremstyle{definition}
\newtheorem{definition}[theorem]{Definition}
\newtheorem*{definition*}{Definition}

\numberwithin{equation}{section}

\newcount\quantno
\everydisplay\expandafter{\the\everydisplay\quantno=0}
\everycr\expandafter{\the\everycr\quantno=0}
\newcommand\quant{\advance\quantno by1
                      \ifnum\quantno=1\qquad\else\quad\fi\forall }

\renewcommand\mod[1]{\vert{#1}\vert}
\newcommand\bigmod[1]{\bigl\vert{#1}\bigr|}

\newcommand\norm[2]{{\Vert{#1}\Vert_{#2}}}

\newcommand\al{\alpha}

\newcommand\ga{\gamma}    
\newcommand\de{\delta}

\newcommand\la{\lambda}   
\newcommand\om{\omega}

\newcommand\hu{H^1}
\newcommand\frh{\mathfrak{h}} 
\newcommand\ghu{{\frh}^1}
\newcommand\ghuH{\frh_{\sgrH}^1}
\newcommand\ghuP{\frh_{\sgrP}^1}

\newcommand\wrt{\,\mathrm{d}}
\newcommand\dtt[1]{\,\frac{\mathrm{d} #1}{ #1}}

\DeclareMathOperator{\Card}{Card}

\newcommand\bBall{\mathbf{B}}

\newcommand\Balls{\mathscr{B}}  

\newcommand{\Discr}{\mathbf{P}}

\newcommand\sgrH{\mathscr{H}}  
\newcommand\sgrP{\mathscr{P}} 

\newcommand\opId{\mathscr{I}} 

\newcommand{\rzP}{\mathbf{I}}

\newcommand\opL{\mathscr{L}}
 
\newcommand\cutL{\mathbf{L}}

\newcommand\maxM{\mathscr{M}}

\newcommand\aiV{\mathscr{V}} 
\newcommand{\taiV}{\wt{\aiV}}

\newcommand\opK{\mathscr{K}}  

\newcommand{\opS}{\mathscr{S}}

\newcommand{\opT}{\mathscr{T}}

\newcommand{\defeq}{\mathrel{:=}}

\newcommand\wt{\widetilde}

\newcommand\e{\mathrm{e}}

\DeclareMathOperator{\sgn}{sgn}
\DeclareMathOperator{\supp}{supp}

\newcommand{\cutF}{\mathscr{F}}
\newcommand{\maxG}{\mathscr{G}}

\newcommand{\tc}{\mathop{\colon}}

\newcommand{\RR}{\mathbb{R}}
\newcommand{\CC}{\mathbb{C}}
\newcommand{\NN}{\mathbb{N}}

\newcommand{\chr}{\mathbf{1}}

\newcommand{\loc}{\mathrm{loc}}

\DeclareMathOperator*{\esssup}{ess\,sup}

\def\Xint#1{\mathchoice
   {\XXint\displaystyle\textstyle{#1}}%
   {\XXint\textstyle\scriptstyle{#1}}%
   {\XXint\scriptstyle\scriptscriptstyle{#1}}%
   {\XXint\scriptscriptstyle\scriptscriptstyle{#1}}%
   \!\int}
\def\XXint#1#2#3{{\setbox0=\hbox{$#1{#2#3}{\int}$}
     \vcenter{\hbox{$#2#3$}}\kern-.5\wd0}}

\def\dashint{\Xint-}

\begin{document}

\title[Maximal characterisation of local Hardy spaces]{Maximal characterisation 
of local Hardy spaces \\ on locally doubling manifolds}  



\thanks{Work partially supported by PRIN 2015 ``Real and complex manifolds: 
geometry, topology and harmonic analysis''.
The first- and third-named authors are members of the Gruppo Nazionale per l'Analisi Matematica, 
la Probabilit\`a e le loro Applicazioni (GNAMPA) of the Istituto 
Nazionale di Alta Matematica (INdAM)}

\author[A. Martini]{Alessio Martini}
\address[Alessio Martini]{
School of Mathematics \\ University of Birmingham\\ Edgbaston, Birmingham, B15 2TT  \\ United Kingdom}
\email{a.martini@bham.ac.uk}

\author[S. Meda]{Stefano Meda}
\address[Stefano Meda]{
Dipartimento di Matematica e Applicazioni
\\ Universit\`a di Milano-Bicocca\\
via R.~Cozzi 53\\ I-20125 Milano\\ Italy}
\email{stefano.meda@unimib.it}

\author[M. Vallarino]{Maria Vallarino}
\address[Maria Vallarino]{
Dipartimento di Scienze Matematiche 
"Giuseppe Luigi Lagrange", Dipartimento di Eccellenza 2018-2022
\\ Politecnico di Torino\\
corso Duca degli Abruzzi 24\\ 10129 Torino\\  Italy}
\email{maria.vallarino@polito.it}

\begin{abstract}
We prove a radial maximal function characterisation of the local atomic Hardy space
$\ghu(M)$ on a Riemannian manifold $M$ with positive injectivity radius and Ricci curvature bounded from below. 
As a consequence, we show that an integrable function belongs to $\ghu(M)$ if and only if either
its local heat maximal function or its local Poisson maximal function are integrable.  
A key ingredient is a decomposition of H\"older cut-offs in terms 
of an appropriate class of approximations of the identity, which we obtain on arbitrary Ahlfors-regular metric measure spaces and generalises a previous result of A.~Uchiyama.
\end{abstract}

\maketitle

\section{Introduction} \label{s:Introduction}

D.~Goldberg \cite{G} introduced a ``local'' Hardy space $\ghu(\RR^n)$, which may be
defined in several equivalent ways:
\begin{equation} \label{f: ghuBRn}
	\begin{split}
\ghu(\RR^n)
		& = \left\{f \in L^1(\RR^n) \tc \mod{\nabla (\opId-\Delta)^{-1/2} f} \in L^1(\RR^n) \right\} \\
		& = \left\{f \in L^1(\RR^n) \tc {\sup_{0<t\leq 1}\, \mod{\sgrH_t f}} \in L^1(\RR^n) \right\} \\
		& = \left\{f \in L^1(\RR^n) \tc {\sup_{0<t\leq 1}\, \mod{\sgrP_t f}} \in L^1(\RR^n) \right\},
	\end{split}
\end{equation}
where $\opId$ denotes the identity map, $\Delta$ the standard Laplacian, $\sgrH_t$ the heat semigroup $\e^{t\Delta}$
and $\sgrP_t$ the Poisson semigroup $\e^{-t\sqrt{-\Delta}}$.   Furthermore, $\ghu(\RR^n)$
admits both an atomic and a ionic decomposition, and can be characterised in terms
of a suitable ``grand maximal function''.

The main advantage of working with $\ghu(\RR^n)$ rather than with the classical Hardy space $\hu(\RR^n)$ \cite{FS}
is that $\ghu(\RR^n)$ is preserved by multiplication by smooth functions with compact support.  This makes 
$\ghu(\RR^n)$ very effective in many situations in which localisation arguments are involved.  

Analogues of $\ghu(\RR^n)$ may be defined in a variety of settings.  In particular, all the 
definitions mentioned above in the Euclidean case make sense on any (complete) Riemannian manifold~$M$, the role of $\Delta$
being played by the Laplace--Beltrami operator on $M$.  
It is then natural to speculate whether all such definitions give rise to the same space.  
Even a bare knowledge of the theory of $\ghu(\RR^n)$
suggests that the key properties of this space depend mainly on the local structure of the Euclidean space.  This leads
to conjecture that a theory parallel to that in $\RR^n$ should hold on any Riemannian manifold where
the local geometry is somewhat uniformly controlled.  

A number of results in this direction are available in the case where the manifold is \emph{doubling}. Indeed, an extensive theory of local Hardy spaces has been developed in the general context of doubling metric measure spaces (see, e.g., \cite{HMY,YZ} and references therein), and includes both atomic and maximal characterisations (at least, under assumptions such as the ``reverse doubling'' condition). This theory is somewhat parallel to that of the ``global'' Hardy space $H^1$ \`a la Coifman--Weiss \cite{CW} on spaces of homogeneous type. However, due to the aforementioned local nature of $\ghu$, a global assumption such as the doubling condition does not appear entirely natural for its study, and one may expect that a richer theory could be developed, also encompassing non-doubling manifolds.

This problem
 has been considered by M.~Taylor \cite{T}, who introduced a local Hardy space $\ghu(M)$ on Riemannian $n$-manifolds $M$ with \emph{strongly bounded geometry} 
(positive injectivity radius and uniform control of all the derivatives of the metric tensor) via a grand maximal function characterisation;
 more precisely, Taylor defines
\begin{equation}\label{f: grandmaxchar}
\ghu(M)=\left\{f\in L^1_\loc(M) \tc \maxG^b  f\in L^1(M) \right\} ,
\end{equation}
where
\begin{equation}\label{f: grandmax}
\maxG^b f(x) 
\defeq \sup_{r \in (0,1]} \sup_{\phi\in \cutL(x,r)} \left|\int_M \phi f\wrt\mu\right|
,
\end{equation}
$\mu$ is the Riemannian measure and, 
 for every $x\in M$ and $r\in (0,1]$, $\cutL(x,r)$ is the collection 
of all $C^1$ functions on $M$ with Lipschitz constant at most $r^{-(n+1)}$ supported in the ball of centre $x$ and radius $r$.
 Further extensions of the theory are due to S.~Volpi and the second-named author \cite{Vo,MVo}, who studied an atomic local Hardy space in the more general context of \emph{locally doubling} metric measure spaces (see Section \ref{s: Background material} below for additional details).
 The atomic space of \cite{MVo} coincides with the space of \cite{T} in the case of manifolds with strongly bounded geometry (see Remark \ref{rem:grandmax_atomic} below), and also with that of \cite{YZ} in the case of doubling spaces. These works, however, do not address the issue of whether the local Hardy space admits characterisations analogous to \eqref{f: ghuBRn} in a non-doubling setting.

Our research was motivated by the following simple question.  Suppose that $M$ is a Riemannian manifold of dimension $n$,
denote by $\opL$ the (positive) Laplace--Beltrami operator on $M$, consider the
associated heat and Poisson semigroups, namely $\sgrH_t \defeq \e^{-t\opL}$ and $\sgrP_t \defeq \e^{-t\sqrt \opL}$,
and the spaces 
\begin{align*}
\ghuH(M)
	& \defeq \left\{f \in L^1(M) \tc \sup_{0<t\leq 1} \mod{\sgrH_t f} \in L^1(M) \right\}, \\
\ghuP(M)
	& \defeq \left\{f \in L^1(M) \tc \sup_{0<t\leq 1} \mod{\sgrP_t f} \in L^1(M) \right\}.
\end{align*}
What geometric assumptions are needed in order that
\begin{equation}\label{eq:question}
\ghu(M) = \ghuH(M) = \ghuP(M),
\end{equation}
where $\ghu(M)$ denotes the atomic local Hardy space of \cite{MVo,T}?

Despite our efforts, we have not been able to find in the literature a proof of the equivalence 
of $\ghu(M)$, $\ghuH(M)$ and $\ghuP(M)$ on a general class of noncompact manifolds extending beyond the doubling ones.  
As suggested above, some ``uniformity'' of
the local geometry should be the essential feature of $M$ in order that the desired equivalence hold.  
One of our main results states that, if $M$ has \emph{bounded geometry}, viz.\ positive injectivity radius and Ricci curvature bounded from below (a weaker assumption than that of \cite{T}), then 
indeed \eqref{eq:question} holds true.

As a matter of fact, for the same class of manifolds $M$, we prove a much more general characterisation of $\ghu(M)$ in terms of an arbitrary ``radial maximal function'', associated to a family of integral operators
\[
\opK_t f(x) = \int_M K(t,x,y) \, f(y) \wrt \mu(y), \qquad t \in (0,1],
\]
whose integral kernel $K$ satisfies suitable assumptions. Roughly speaking (see Section \ref{s:manifold_maximal} below for details), we require $K$ to decompose as the sum $K^0 + K^\infty$, where the local part $K^0$ is supported in a $t$-independent neighbourhood of the diagonal and satisfies bounds of the form
\begin{equation}\label{eq:local_bds_ai}
0 \leq K^0(t,x,y) \leq C t^{-n} (1+d(x,y)/t)^{-n-\gamma}, \qquad K^0(t,x,x) \geq c t^{-n}
\end{equation}
and a $\gamma$-H\"older condition for some $\gamma \in (0,1]$, while the global part $K^\infty$ satisfies the ``uniform integrability'' condition
\begin{equation}\label{eq:global_bds_ai}
\sup_{x \in M} \int_M \sup_{0 < t \leq 1} |K^\infty(t,x,y)| \wrt \mu(y) < \infty.
\end{equation}
Here $d$ and $\mu$ denote the Riemannian distance and measure on $M$ respectively. Under these assumptions on $K$, we prove the maximal characterisation
\begin{equation}\label{eq:max_char}
\ghu(M) = \left\{ f \in L^1(M) \tc \sup_{0 < t \leq 1} |\opK_t f| \in L^1(M) \right\}
\end{equation}
for all manifolds $M$ with bounded geometry.

Bounds similar to \eqref{eq:local_bds_ai} are considered in many places in the literature, including the aforementioned works \cite{HMY,YZ} on local Hardy spaces in a doubling context, where various classes of kernels (dubbed ``approximations of the identity'') are considered. It should be pointed out that \cite{HMY,YZ} do not require nonnegativity or on-diagonal lower bounds as in \eqref{eq:local_bds_ai}; however, they impose, \emph{inter alia}, a normalisation condition of the form
\begin{equation}\label{eq:normalisation}
\int_M K(t,x,z) \wrt \mu(z) = 1 = \int_M K(t,z,x) \wrt \mu(z) \quant x \in M \quant t \in (0,1].
\end{equation}
Such a condition is quite delicate, in the sense that it is not generally preserved by localisation procedures, such as multiplication by cutoff functions, or changes of variables (involving a change of measure). Indeed, the very construction of ``approximations of the identity'' satisfying \eqref{eq:normalisation} on arbitrary doubling spaces is itself not a trivial matter (see, e.g., \cite[Theorem 2.6]{HMY}). In these respects, our assumptions on the kernel, which do not include \eqref{eq:normalisation}, appear to be more robust in nature, and this feature actually turns out to be essential for our proof.

Indeed, in order to prove the maximal characterisation of $\ghu(M)$ in terms of a given kernel $K$, we reduce through a localisation argument (partly inspired by ideas in \cite{T}) to proving the analogous characterisation of $\ghu(\RR^n)$ in terms of localised versions of $K$. However, even if we start with a particularly well-behaved kernel $K$ on $M$ (such as the heat or Poisson kernels), which satisfies the normalisation condition \eqref{eq:normalisation}, there is no reason why the resulting localised kernels on $\RR^n$ would have the same property. Hence, maximal characterisations such as that in \cite{YZ} do not appear to directly apply to the problem at hand.

Instead, here we resort to a different approach, based on a deep result of A.\ Uchiyama \cite{U} (see also \cite[pp.\ 641--642]{CW} for an antecedent of the result).
In \cite{U}, among other things, a maximal characterisation for the Coifman--Weiss Hardy space $H^1$ is proved in the context of \emph{Ahlfors-regular} metric measure spaces (a subclass of doubling spaces including $\RR^n$), in terms of arbitrary kernels $K$ satisfying pointwise bounds analogous to \eqref{eq:local_bds_ai}, \emph{without any normalisation assumptions}. Roughly speaking, the approach in \cite{U} goes by showing that any $\gamma$-H\"older cutoff can be written as an infinite linear combination $\sum_j c_j K(t_j,x_j,\cdot)$ for appropriate choices of coefficients $c_j$, times $t_j$ and points $x_j$; this decomposition in turn yields a majorisation of the grand maximal function defined via $\gamma$-H\"older cutoffs in terms of the radial maximal function associated to $K$.

Actually Uchiyama's result does not directly apply to our setting, since he works with $H^1$ instead of $\ghu$, and correspondingly he considers ``global'' kernels $K(t,x,y)$ defined for all $t \in (0,\infty)$; as a matter of fact, the decomposition given in \cite{U} of a given $\gamma$-H\"older cutoff may include terms $K(t_j,x_j,\cdot)$ with $t_j$ arbitrarily large. A contribution of the present paper, which may be of independent interest, is the adaptation of Uchiyama's argument to the case of local Hardy spaces and kernels, which is presented in Section \ref{s:Uchiyama} below in the generality of Ahlfors-regular metric spaces. Differently from \cite{U}, here we provide a decomposition of $\gamma$-H\"older cutoffs at scale $s$ that only employs terms $K(t_j,x_j,\cdot)$ with $t_j \leq s$; this allows us to work with kernels $K(t,x,y)$ only defined for $t \in (0,1]$, since in the case of local Hardy spaces we are only interested in small scales. This variant of Uchiyama's result is the crucial ingredient that allows us to close the localisation argument and prove in Section \ref{s:manifold_maximal} the maximal characterisation \eqref{eq:max_char} on manifolds with bounded geometry.

By using well-known Gaussian-type heat kernel bounds for small times, one can readily check that the heat and Poisson kernels on a manifold $M$ with bounded geometry satisfy the assumptions \eqref{eq:local_bds_ai} and \eqref{eq:global_bds_ai}. Indeed, in Section \ref{s:heat_poisson}, we show that this is the case more generally for the semigroups $\e^{-t\opL^\alpha}$ with $\alpha \in (0,1]$, thus obtaining the characterisation
\[
\ghu(M) = \left\{f \in L^1(M) \tc \sup_{0<t\leq 1} \mod{\e^{-t\opL^\alpha} f} \in L^1(M) \right\}, 
\]
which includes \eqref{eq:question} as a special case.

\medskip

The present paper does not address the problem whether $\ghu(M)$ also admits a local Riesz transform characterisation, analogous to the first identity in \eqref{f: ghuBRn} for the case of $\ghu(\RR^n)$. This deceptively simple question turns out to require a much more sophisticated analysis, and is solved in the affirmative in a recent work of G.~Veronelli and the second-named author \cite{MVe}. One of the ingredients used in \cite{MVe} is the Poisson maximal characterisation $\ghu(M) = \ghuP(M)$ that we prove here.

Another question that we do not address here is the investigation of spaces defined in terms of ``global'' maximal functions, such as
\begin{align*}
H^1_\sgrH(M) &= \left\{ f \in L^1(M) \tc \sup_{0<t<\infty} |\sgrH_t f| \in L^1(M) \right\}, \\
H^1_\sgrP(M) &= \left\{ f \in L^1(M) \tc \sup_{0<t<\infty} |\sgrP_t f| \in L^1(M) \right\},
\end{align*}
in the context of a nondoubling manifold $M$. Nevertheless, the results in the present paper turn out to be instrumental in the analysis of such spaces and their relation to the Hardy-type spaces $\mathfrak{X}^\gamma(M)$ introduced in \cite{Vo,MMV}, which we plan to develop in a future work \cite{MMVV}.

\medskip

It is an interesting question whether the maximal characterisation \eqref{eq:max_char} extends to larger classes of Riemannian manifolds, or even more general spaces. A particularly natural setting for this investigation would be that of the locally doubling metric measure spaces considered in \cite{MVo}. Given that our core ingredient (the local variant of Uchiyama's result) is proved for general Ahlfors-regular spaces, it would seem natural to conjecture that a maximal characterisation of $\ghu$ in terms of a single kernel holds at least on metric measure spaces satisfying a suitable ``local Ahlfors'' condition. Extensions to even broader classes of spaces may also be possible; however, even in the case of globally doubling spaces, radial maximal characterisations of $H^1$ and $\ghu$ appear to be available only under additional assumptions, such as a reverse doubling condition (see, e.g., \cite{GLY,YZ,YZ2}), so tackling the general case of locally doubling spaces may be a nontrivial problem.

\medskip 

We shall use
the ``variable constant convention'', and denote by $C$
a constant that may vary from place to
place and may depend on any factor quantified (implicitly or explicitly)
before its occurrence, but not on factors quantified afterwards.

\section{Background on Hardy-type spaces}
\label{s: Background material}

Let $M$ denote a connected, complete $n$-dimensional Riemannian manifold
 with Riemannian measure $\mu$ and Riemannian distance $d$.  
Throughout this paper we assume that $M$ has \emph{bounded geometry}, that is,
\begin{enumerate}[label=(\Alph*)]
\item\label{en:ass_inj} the injectivity radius $\iota_M$ of $M$ is positive;
\item\label{en:ass_ric} the Ricci tensor of $M$ is bounded from below. 
\end{enumerate}

For $p$ in $[1,\infty]$, $\|f\|_p$ denotes the $L^p(M)$ norm of $f$ (with respect to the Riemannian measure).
 
We denote by $\Balls$ the family of all geodesic balls on $M$.
For each $B$ in $\Balls$ we denote by $c_B$ and $r_B$
the centre and the radius of $B$ respectively.  
Furthermore, we denote by $c \, B$ the
ball with centre $c_B$ and radius $c \, r_B$.
For each \emph{scale parameter} $s$ in $\RR^+$, 
we denote by $\Balls_s$ the family of all
balls $B$ in $\Balls$ such that $r_B \leq s$.  We also write $B_r(x)$ for the geodesic ball of centre $x \in M$ and radius $r > 0$.

For manifolds satisfying \ref{en:ass_inj}-\ref{en:ass_ric} above,
the Bishop--Gromov theorem (see, e.g., \cite[Theorem III.4.4]{Ch} or \cite[Section 5.6.3]{SC}) and results by Anderson--Cheeger on harmonic coordinates (see, for instance, \cite[Theorem~1.2]{He})
imply the following properties.
\begin{enumerate}[label=(\alph*)]
\item\label{en:localahlfors}
$M$ is \emph{uniformly locally $n$-Ahlfors}, i.e.\ for every $s>0$ there exists a positive constant $D_s$ such that 
\begin{equation} \label{f: local Ahlfors}
D_s^{-1}\,r_B^n
\leq \mu(B)
\leq D_s\,r_B^n
\quant B\in\Balls_s;
\end{equation}
in particular, $M$ satisfies the \emph{local doubling property}, i.e.\ for every $s > 0$ there exists a positive constant $D_s'$ such that
\begin{equation}\label{f: local doubling}
\mu(2B) \leq D_s' \, \mu(B) \quant B \in \Balls_s.
\end{equation}
\item\label{en:exp_growth}
$M$ has \emph{at most exponential growth}, i.e.\ there exist positive constants $a$ and $\beta$ such that
\begin{equation}\label{eq:exp_growth}
\mu(B) \leq a \exp(\beta r_B)  \quant B \in \Balls.
\end{equation}
\item\label{en:harmonic_radius}
For all $Q>1$ and $\alpha \in (0,1)$,
the $(Q^2,0,\alpha)$-\emph{harmonic radius} $r_H$ of $M$ is strictly positive.  In particular,
to each point $x$ in $M$ we can associate 
a harmonic co-ordinate system $\eta_x$ centred at $x$ and defined on $B_{r_H}(x)$ such that, in these coordinates, the metric tensor $(g_{ij})$ satisfies the estimate
\begin{equation} \label{f: control of d}
(\de_{ij})/Q^2 \leq (g_{ij}) \leq Q^2\, (\de_{ij}) \text{ as quadratic forms}
\end{equation}
at every point of $B_{r_H}(x)$. In particular, as a consequence,
\begin{equation} \label{f: control of d II}
\mod{Y-Z}/Q \leq d(y,z) \leq Q\, \mod{Y-Z}
\quant y,z \in B_{r_H/2}(x) \quant x\in M,
\end{equation}
where $Y = \eta_x(y)$ and $Z = \eta_x(z)$.
Moreover \eqref{f: control of d} 
implies that
\begin{equation}\label{f: control of d III}
Q^{-n} \leq \sqrt{\,\overline{g}\,} \leq Q^{n} \,,
\end{equation}
where $\overline{g}$ denotes the determinant of the metric tensor.  
\end{enumerate}

We point out that the key aspect of the estimates \ref{en:localahlfors}-\ref{en:harmonic_radius} is their uniformity with respect to the centre of the ball $B$ or the point $x$. Indeed, if one does not care about uniformity, then estimates similar to those in \ref{en:localahlfors} and \ref{en:harmonic_radius} are easy consequences of the properties of normal coordinates (see, e.g., \cite[Proposition 5.11]{Lee}), and do not require the assumptions \ref{en:ass_inj}-\ref{en:ass_ric}. This includes, for all $x \in M$, the volume asymptotics
\begin{equation}\label{eq:vol_asym}
\mu(B_r(x)) = \omega_n r^n (1+o(1)) \quad\text{as } r \to 0^+,
\end{equation}
where $\omega_n$ is the volume of the unit ball in Euclidean $n$-space.

\bigskip

We now recall the definition of the atomic local Hardy space $\ghu(M)$. This is a particular instance of the local Hardy space introduced by Volpi \cite{Vo}, who extended previous 
work of Taylor \cite{T}, and then further generalised in \cite{MVo}.

\begin{definition} \label{d: atom}
Fix $s>0$.  Suppose that $p$ is in $(1,\infty]$ and let $p'$ be the index conjugate to $p$. 
A \emph{standard $p$-atom at scale $s$} 
is a function $a$ in $L^1(M)$ supported in a ball $B$ in $\Balls_s$ 
satisfying the following conditions:
\begin{enumerate}[label=(\roman*)]
\item
\emph{size condition}: $\norm{a}{p}  \leq \mu (B)^{-1/p'}$;
\item \emph{cancellation condition}:  
$\int_B a \wrt \mu  = 0$. 
\end{enumerate}
A \emph{global $p$-atom} at scale $s$ is a function $a$
in $L^1(M)$ supported in a ball $B$ of radius \emph{exactly equal to} $s$ 
satisfying the size condition above (but possibly not the cancellation condition).
Standard and global $p$-atoms at scale $s$ will be referred to simply as $p$-\emph{atoms} at scale $s$.
\end{definition}

\begin{definition} \label{d: Goldberg}
The \emph{local atomic Hardy space} $\frh_s^{1,p}({M})$ is the 
space of all functions~$f$ in $L^1(M)$
that admit a decomposition of the form
\begin{equation} \label{f: decomposition}
f = \sum_{j=1}^\infty \la_j \, a_j,
\end{equation}
where the $a_j$'s are $p$-atoms at scale $s$ and $\sum_{j=1}^\infty \mod{\la_j} < \infty$.
The norm $\norm{f}{\frh_s^{1,p}}$ of $f$ is the infimum of $\sum_{j=1}^\infty \mod{\la_j}$
over all decompositions (\ref{f: decomposition}) of $f$. 
\end{definition}

One can prove (see, for instance, \cite{MVo}) that $\frh_s^{1,p}$ is independent of both  
$s$ in $(0,\infty)$ and $p$ in $(1,\infty]$ (in the sense that different choices of the parameters define equivalent norms); henceforth it will be denoted simply by $\ghu(M)$, and $\|f\|_{\ghu}$ will denote the norm $\|f\|_{\frh_1^{1,2}}$. We will also say ``$\ghu(M)$-atom at scale $s$'' instead of ``$2$-atom at scale $s$''.

The following statement will be useful in proving boundedness properties of sublinear operators defined on $\ghu(M)$.

\begin{lemma}\label{lemma:sublinear_hardy}
Let $\opT : L^1(M) \to L^{1,\infty}(M)$ be a bounded sublinear operator. Let $p \in (1,\infty]$ and $s>0$. If
\[
\sup \left\{ \| \opT a \|_{L^1(M)} \tc a \text{ $p$-atom at scale $s$ on $M$} \right\} < \infty,
\]
then $\opT$ maps $\ghu(M)$ into $L^1(M)$ boundedly.
\end{lemma}
\begin{proof}
Suppose that $f$ is in $\ghu(M)$, and write $f= \sum_{j=1}^\infty \lambda_j \, a_j$, where $a_j$ are
$p$-atoms at scale $s$, and $\sum_{j=1}^\infty |\lambda_j| < \infty$.

 For each positive integer $N$, write 
$f_N = \sum_{j=1}^N \lambda_j\, a_j$ and note that $f_N$ tends to $f$ in $\ghu(M)$.
Since $\ghu(M)$ is continuously contained in $L^1(M)$, $\opT f$ is a well defined element of the Lorentz space 
$L^{1,\infty}(M)$, and $\opT f_N$ tends to $\opT f$ in $L^{1,\infty}(M)$.

On the other hand, by sublinearity, if $N'>N$, then
\[
|\opT f_N - \opT f_{N'}| \leq |\opT (f_N-f_N')| \leq \sum_{j=N+1}^{N'} |\lambda_j| |\opT a_j|,
\]
so from the uniform boundedness of $\opT$ on atoms and the convergence of the series $\sum_j |\lambda_j|$ we deduce that $\opT f_N$ is a Cauchy sequence in $L^1(M)$. By uniqueness of limits, we conclude that $\opT f_N$ converges to $\opT f$ in $L^1(M)$, and
\[
\|\opT f\|_1 = \lim_{N \to \infty} \left\|\opT f_N\right\|_{1} \leq \sum_{j=1}^\infty |\lambda_j| \left\|\opT a_j\right\|_1 \leq C \sum_{j=1}^\infty |\lambda_j|.
\]
By taking the infimum of both sides with respect to all the representations of $f$ as sum of $p$-atoms,
we obtain that 
$\norm{\opT f}{1} \leq C \norm{f}{\frh^{1,p}_{s}(M)}$,
as required.  
\end{proof}

The space $\ghu(M)$ can also be characterised in terms of ions as follows.   

\begin{definition}\label{d: ions}
Suppose that $s>0$, $p$ is in $(1,\infty]$ and let $p'$ be the index conjugate to~$p$. 
A \emph{$p$-ion} at scale $s$ is a function $g$ in $L^1(M)$ supported in a ball $B$ in $\Balls_s$ 
satisfying the following conditions:
\begin{enumerate}[label=(\roman*)]
\item\label{en:ions_size}
$\norm{g}{p}  \leq \mu (B)^{-1/p'}$;
\item\label{en:ions_canc}
$\left|\int_B g \wrt \mu\right|   \leq r_B$.
\end{enumerate}
\end{definition}

\begin{definition} \label{d: ionic}
The \emph{local ionic Hardy space} $\frh_{s,I}^{1,p}({M})$ is the 
space of all functions~$f$ in $L^1(M)$ that admit a decomposition of the form
\begin{equation} \label{f: ionicdecomposition}
f = \sum_{j=1}^\infty \la_j \, g_j,
\end{equation}
where the $g_j$'s are $p$-ions at scale $s$ and $\sum_{j=1}^\infty \mod{\la_j} < \infty$.
The norm $\norm{f}{\frh^{1,p}_{s,I}}$ of $f$ is the infimum of $\sum_{j=1}^\infty \mod{\la_j}$
over all decompositions (\ref{f: ionicdecomposition}) of $f$.
\end{definition}

It was proved in \cite[Theorem 1]{MVo} that for every $s>0$ and $p$ in $(1,\infty]$ the spaces 
$\frh^{1,p}_I({M})$ coincide with $\ghu(M)$. Moreover, for each $s>0$ and $p \in (1,\infty]$ there exists a positive constant $C$ such that 
\begin{equation}\label{f: equivionicatomic}
C^{-1} \norm{f}{\frh^{1,p}_{s,I}}
\leq \norm{f}{\ghu} 
\leq C \norm{f}{\frh^{1,p}_{s,I}}
\quant f\in \ghu(M)\,.
\end{equation}

The cancellation condition that standard atoms must satisfy is in general not preserved by changes of variables and localisations; however, as shown in the following lemma, performing such operations on an atom produces an ion (or a multiple thereof). This observation, together with the equivalence \eqref{f: equivionicatomic}, confirms that $\ghu(M)$ is amenable to localisations and changes of variables.

The following statement involves two Riemannian manifolds $M$ and $M'$, both satisfying the assumptions \ref{en:ass_inj}-\ref{en:ass_ric} above; correspondingly, we denote by $d$ and $d'$ the respective Riemannian distances, and by $\mu$ and $\mu'$ the Riemannian measures. The result is certainly known to experts, and implicit in the work of Taylor \cite{T}, under more restrictive assumptions on $M$ and $M'$.

\begin{lemma}\label{lem:atom_ion}
Let $p \in (1,\infty]$, $s,L>0$, and $A \geq 1$. 
Let $\phi : M \to \CC$ satisfy
\[
|\phi(x)| \leq L, \qquad |\phi(x)-\phi(y)| \leq L d(x,y)
\]
for all $x,y \in M$. 
Let $\Omega$ and $\Omega'$ be open subsets of $M$ and $M'$, and let $\Psi : \Omega' \to \Omega$ be a bi-Lipschitz map such that the Lipschitz constants of $\Psi$ and $\Psi^{-1}$ are both bounded by $A$. Let $\rho : \Omega' \to (0,\infty)$ be the density of the push-forward of $\mu$ via $\Psi^{-1}$ with respect to $\mu'$.
Then there exists a constant $H$, only depending on $M$, $M'$, $p$, $s$, $L$ and $A$, such that, 
for every $p$-atom $a$ at scale $s$ on $M$, supported in a ball $B \subseteq \Omega$,
if $g : M' \to \CC$ is defined by
\[
g(x') = \begin{cases}
\rho(x') \phi(\Psi(x')) a(\Psi(x'))/H &\text{if } x' \in \Omega',\\
0 &\text{otherwise},
\end{cases}
\]
then $g$ is a $p$-ion on $M'$ at scale $A s$. 
\end{lemma}
\begin{proof}
Let $a$ be a $p$-atom at scale $s$ on $M$, supported in a ball $B \subseteq \Omega$. Let $H>0$ be a positive constant and define $g$ as above. In the course of the proof, we will determine what conditions $H$ must satisfy in order for the above statement to hold.

Let $B'$ be the ball on $M'$ of centre $\Psi^{-1}(c_B)$ and radius $Ar_B$. Then clearly $\Psi^{-1}(B) \subseteq B' \cap \Omega'$ and $g$ is supported in $B'$. Moreover
\[
\int_{B'} g \wrt\mu' = H^{-1} \int_{\Omega' \cap B'} \phi(\Psi(x')) a(\Psi(x')) \rho(x') \wrt\mu'(x') = H^{-1} \int_{B} \phi  a  \wrt\mu.
\]
If $a$ is a standard atom, then
\[
\left| \int_B \phi a \wrt\mu \right| = \left| \int_B (\phi - \phi(c_B)) a \wrt\mu \right| \leq L r_B  \|a\|_1 \leq L r_B,
\]
so the condition \ref{en:ions_canc} in Definition \ref{d: ions} is satisfied provided 
$H \geq L/A$. If instead $a$ is a global atom, then $r_B = s$ and
\[
\left| \int_B \phi a \wrt\mu \right| \leq L \|a\|_1 \leq L ,
\]
so the condition is satisfied provided
$H \geq L/(As)$.

As for the condition \ref{en:ions_size} of Definition \ref{d: ions}, let us first notice that,
 for all $x' \in \Omega'$,
\[
\rho(x') = \lim_{r \to 0^+} \frac{\mu(\Psi(B'_r(x'))}{\mu'(B'_r(x'))} \leq  \lim_{r \to 0^+} \frac{\mu(B_{Ar}(\Psi(x')))}{\mu'(B'_r(x'))} = A^n; 
\]
the latter equality is a consequence of \eqref{eq:vol_asym}.
Hence the size condition on the $p$-atom $a$ implies that
\[
\|g\|_{L^p(M')} \leq H^{-1} L A^{n/p'} \mu(B)^{-1/p'},
\]
where $p'$ is the conjugate exponent to $p$.
On the other hand, since $M$ and $M'$ are both uniformly locally $n$-Ahlfors, there exists a constant $\kappa \geq 1$, only depending on $M$, $M'$ and $s$, such that
\[
\mu'(B') \leq \kappa r_{B'}^n = \kappa A^n r_B^n \leq \kappa^2 A^n \mu(B),
\]
whence
\[
\|g\|_{L^p(M')} \leq H^{-1} L A^{2n/p'} \kappa^{2/p'} \mu'(B')^{-1/p'}.
\]
So the condition \ref{en:ions_size} of Definition \ref{d: ions} is satisfied provided $H \geq  L A^{2n/p'} \kappa^{2/p'}$.
\end{proof}

\begin{remark}\label{rem:grandmax_atomic}
As mentioned in the introduction, on a manifold $M$ which has 
strongly bounded geometry in the sense of \cite[Conditions (1.21)-(1.23)]{T} 
Taylor defined a local Hardy space by means of the grand maximal function \eqref{f: grandmax},
which turns out to be equivalent to the atomic space $\ghu(M)$ defined above (cf.\ \cite[Section 5]{T} and \cite[Theorem 1]{MVo}).
The results of this paper (see Corollary \ref{cor:grandmaximal} below) can actually be used to show that the grand maximal characterisation \eqref{f: grandmaxchar} of $\ghu(M)$ extends to the generality of the manifolds $M$ considered here.
\end{remark}

\section{Interlude: A result on metric measure spaces}\label{s:Uchiyama}

In this section we prove a variation of a result of Uchiyama \cite{U}, which plays a fundamental role in our proof of the radial maximal characterisation for local Hardy spaces. Differently from the rest of the paper, here we do not work on a Riemannian manifold $M$, but on a metric measure space $X$. Due to the different setting, part of the notation used here differs from that used in other sections.

Let $D \in (0,\infty)$. Let $(X,d,m)$ be a metric measure space which is $D$-\emph{Ahlfors regular}, i.e.\
there exists a constant $A \geq 1$ such that, for all $x \in X$ and $r \in [0,\infty)$, 
\begin{equation}\label{eq:ahlfors}
A^{-1} r^D \leq m(B(x,r)) \leq A r^D;
\end{equation}
here $B(x,r)$ denotes the ball of centre $x$ and radius $r$ in $X$. Lebesgue spaces $L^p(X)$ on $X$ are meant with respect to the measure $m$, and $\|f\|_p$ will denote the $L^p(X)$ norm (or quasinorm, if $p < 1$) of $f$.

The next definition closely follows \cite[eqs.\ (40)-(43)]{U}.

\begin{definition} \label{def: AI Ahlfors}
Let $\gamma \in (0,1]$. An \emph{approximation of the identity} (AI in the sequel) of exponent $\gamma$ on a $D$-Ahlfors regular space $X$
is a measurable function $K : (0,1] \times X \times X \to [0,\infty)$ such that 
for some $c \in (0,1)$, for all $t \in (0,1]$ and $x,y,z \in X$ such that $4d(y,z) \leq t+d(x,y)$,
\begin{align}
K(t,x,y) 
	&\leq t^{-D} \, (1+d(x,y)/t)^{-D-\ga}, \label{eq:Kupper} \\
K(t,x,x) 
	&\geq c \,t^{-D},\label{eq:Klower}\\
|K(t,x,y) - K(t,x,z)| 
	&\leq t^{-D} (d(y,z)/t)^\ga \, (1+d(x,y)/t)^{-D-2\ga}. \label{eq:Klip}
\end{align}
\end{definition}

\begin{remark}
The bounds \eqref{eq:Kupper} and \eqref{eq:Klip} can be equivalently rewritten as
\begin{align*}
K(t,x,y) 
	&\leq t^{\gamma} \, (t+d(x,y))^{-D-\ga}, \\ |K(t,x,y) - K(t,x,z)| &\leq (td(y,z))^\ga \, (t+d(x,y))^{-D-2\ga};
\end{align*}
in the case $\gamma=1$ and $X=\RR^D$, these bounds are clearly satisfied by $K(t,x,y) = t (t^2+|x-y|^2)^{-(D+1)/2}$, which is a constant multiple of the Poisson kernel.
\end{remark}

In the course of this section the exponent $\gamma \in (0,1]$ will be thought of as fixed.

\begin{remark}
If $K$ is an AI, then there exist $c_1,c_2 \in (0,1)$ such that, 
for all $t \in (0,1]$ and $x,y \in X$,
\begin{equation}\label{eq:nonvanishing}
K(t,x,y) \geq c_1 \, t^{-D} \qquad\text{whenever } d(x,y) \leq c_2\, t
\end{equation}
(more precisely, we can take $c_2 = \min\{(c/2)^{1/\ga},1/4\}$ and $c_1 = c/2$).
\end{remark} 

To a measurable kernel $K : (0,1] \times X \times X \to \CC$, we associate the corresponding integral operators $\opK_t$ for $t \in (0,1]$ and the (local) maximal operator $\opK_*$ defined by
\[
\opK_t f(x) = \int_{X}  K(t,x,y) \, f(y) \wrt m(y), 
		\qquad \opK_* f(x) = \sup_{t \in (0,1]} | \opK_t f(x) |.
\]
We also denote by $\maxM$ the \emph{(global) centred Hardy--Littlewood maximal function}:
\[
\maxM f(x) = \sup_{r \in (0,\infty)} \dashint_{B(x,r)} |f(y)| \,\wrt m(y)
\]
(here $\dashint_B f \,\wrt m = m(B)^{-1} \int_B f \,\wrt m$). As is well known, $\maxM$ is of weak type $(1,1)$ and bounded on $L^p(X)$ for all $p \in (1,\infty]$.

Finally, for all $x \in X$, $r \in (0,\infty)$, let $\cutF_\gamma(x,r)$ be the family of \emph{$\gamma$-H\"older cutoffs} on the ball $B(x,r)$, that is, the collection of all functions $\phi : X \to \RR$ such that, for all $y,z \in X$,
\[
\supp \phi \subseteq B(x,r),  \qquad
|\phi(y)| \leq r^{-D}, \qquad
|\phi(z) - \phi(y)| \leq r^{-D} (d(z,y)/r)^\gamma.
\]
Then we define the \emph{($\gamma$-H\"older) local grand maximal function} $\maxG_\gamma$ by
\begin{equation}\label{eq:def_grandmax}
\maxG_\gamma f(x) = \sup_{r \in (0,1]} \sup_{\phi \in \cutF_\gamma(x,r)} \left| \int_X \phi(y) f(y) \,\wrt m(y) \right| .
\end{equation}

The aim of the present section is the proof of the following result, which is a variation of \cite[Theorem 1$'$]{U} for local maximal functions.

\begin{theorem}\label{thm:main}
Let $K$ be an AI on $X$. Then, there exist $E \in (1,\infty)$ and $p \in (0,1)$ such that, for all $f \in L^1_\loc(X)$,
\begin{equation}\label{eq:pointwise_grand_maxk}
\maxG_\gamma f \leq E (\maxM((\opK_* f)^{p}))^{1/p}
\end{equation}
pointwise. In particular, for all $q \in (p,\infty]$, there exists $E_q \in (0,\infty)$ such that, for all $f \in L^1_\loc(X)$,
\begin{equation}\label{eq:lq_grand_maxk}
\|\maxG_\gamma f\|_{q} \leq E_q \| \opK_* f \|_{q}.
\end{equation}
\end{theorem}

\begin{remark}
As will be clear from the proof, the constants $E$ and $p$ in Theorem \ref{thm:main} only depend on the parameters $A,D,\gamma,c$ in \eqref{eq:ahlfors} and Definition \ref{def: AI Ahlfors}, while $E_q$ only depends on those parameters and $q$. This fact will be crucial in the application of the above result in the following Section \ref{s:manifold_maximal}.
\end{remark}

As in \cite{U}, the key ingredient in the proof of this result is the following decomposition of an arbitrary $\ga$-H\"older 
cutoff $\phi$ supported in a ball of radius $1$ as a superposition of kernels $K(t,x,\cdot)$ at different times $t$ and basepoints $x$. The main difference with respect to the decomposition obtained in \cite[proof of Lemma 3$'$]{U} is that here we only use times $t \leq 1$.

\begin{proposition}\label{prp:phidecomp}
Let $K$ be an AI on $X$. There exist $\delta,\eta \in (0,1)$ and $\kappa,L \in (0,\infty)$ such that
\begin{equation}\label{eq:eta_delta_ineq}
\eta^D < 1-\delta
\end{equation}
and the following hold.
Let $o \in X$, and set $d(x) = 1+d(o,x)$ for all $x \in X$.
Let $f \in L^1_\loc(X)$ be such that $\opK_* f \in L^1_\loc(X)$.
Then, for all $i \in \NN$, there exist a finite index set $J(i)$ and, for all $j \in J(i)$, there exist $x_{ij} \in X$ such that, if $B_{ij} = B(x_{ij}, \eta^{1+i} d(x_{ij}))$, then
\begin{gather}
\eta^{i+1} d(x_{ij}) \leq 1, \label{eq:smalltimes}\\
(\opK_* f(x_{ij}))^{1/2} \leq L \, \dashint_{B_{ij}} (\opK_* f(y))^{1/2} \wrt m(y), \label{eq:average_max} \\
\sup_{x \in X} \sum_{j \in J(i)} \chr_{B_{ij}}(x)  \leq L. \label{eq:finite_overlapping}
\end{gather}
Moreover, for all $\phi \in \cutF_\gamma(o,1)$, there exist $\epsilon_{ij} \in \{-1,0,1\}$ for all $i \in \NN$ and $j \in J(i)$ such that,
for all $x \in X$,
\begin{equation}\label{eq:phidecomp}
\phi(x) = \kappa \sum_{i \in \NN} (1-\delta)^i \sum_{j \in J(i)} \epsilon_{ij} d(x_{ij})^{-\ga/2} \eta^{D(1+i)} K(\eta^{1+i} d(x_{ij}), x_{ij}, x).
\end{equation}
\end{proposition}

We postpone the proof of this decomposition to Section \ref{ss:proof_dec}. Let us first show how to derive the main result from this decomposition.
To this purpose the following lemma, which is a simple adaptation of \cite[Lemma 1]{U}, will be useful.

\begin{lemma}\label{lem:carleson}
Let $\rho \in [0,\infty)$ and $q \in (1,\infty)$. There exists $C_{q,\rho} \in (0,\infty)$ such that the following hold.
Let $\nu$ be a nonnegative measure on $X \times (0,\infty)$ such that, for all $x \in X$ and $r \in (0,\infty)$,
\[
\nu(B(x,r) \times (0,r)) \leq r^{D(1+\rho)}.
\]
Then, for all $f \in L^q(X)$,
\[
\left(\int_{X \times (0,\infty)} \left| \dashint_{B(x,r)} f(y) \wrt m(y) \right|^{q(1+\rho)} \wrt \nu(x,r)\right)^{1/(q(1+\rho))} \leq C_{q,\rho} \|f\|_{q}.
\]
\end{lemma}

By combining Proposition \ref{prp:phidecomp} and Lemma \ref{lem:carleson}, we obtain the following crucial majorisation.

\begin{corollary}\label{cor:est_phi_maxk}
Let $K$ be an AI on $X$. There exist $E \in (1,\infty)$ and $p \in (0,1)$ such that the following hold. Let $o \in X$ and $\phi \in \cutF_\gamma(o,1)$. Then, for all $f \in L^1_\loc(X)$,
\begin{equation}\label{eq:est_phi_maxk}
\left| \int \phi(x) f(x) \wrt m(x) \right| \leq E (\maxM((\opK_* f)^{p})(o))^{1/p}.
\end{equation}
\end{corollary}
\begin{proof}
By Proposition \ref{prp:phidecomp}, we can decompose $\phi$ as in \eqref{eq:phidecomp}. Consequently, by \eqref{eq:average_max},
\begin{equation}\label{eq:first_decomp}
\begin{split}
&\left| \int \phi(x) f(x) \wrt m(x) \right| \\
&\leq \kappa \sum_{i \in \NN,\, j \in J(i)} (1-\delta)^i d(x_{ij})^{-\ga/2} \eta^{D(1+i)} (\opK_* f)(x_{ij}) \\
&\leq \kappa L^{2} \sum_{i \in \NN,\, j \in J(i)}  (1-\delta)^i d(x_{ij})^{-\ga/2} \eta^{D(1+i)} \left( \dashint_{B_{ij}} ((\opK_* f)(y))^{1/2} \wrt m(y) \right)^2 \\
&\leq \frac{\kappa L^{2}}{1-\delta} \sum_{k \in \NN} 2^{-k\ga/2}  \int_{X \times (0,\infty)} \left( \dashint_{B(x,r)} ((\opK_* f)(y))^{1/2} \wrt m(y) \right)^2 \wrt\nu_k(x,r),
\end{split}
\end{equation}
where, for all $k \in \NN$, the measure $\nu_k$ on $X \times (0,\infty)$ is defined  by
\[
\nu_k = \sum_{\substack{i \in \NN,\, j \in J(i)\\ 2^k \leq d(x_{ij}) < 2^{k+1}}} (1-\delta)^{1+i} \eta^{D(1+i)} \delta_{(x_{ij},\eta^{1+i} d(x_{ij}))}
\]
and $\delta_{(x,r)}$ denotes the Dirac measure at $(x,r) \in X \times (0,\infty)$.

Note now that, for all $x \in X$ and $r \in (0,\infty)$,
\[\begin{split}
\nu_k(B(x,r) \times (0,r)) &= \sum_{\substack{i \in \NN,\, j \in J(i) \\ 2^k \leq d(x_{ij}) < 2^{k+1} \\ x_{ij} \in B(x,r), \, \eta^{1+i} d(x_{ij}) < r }} (1-\delta)^{1+i} \eta^{D(1+i)} \\
&\leq C \sum_{\substack{i \in \NN \\ \eta^{1+i} 2^k < r}} \left( \frac{r}{\eta^{1+i} 2^k} \right)^D (1-\delta)^{1+i} \eta^{D(1+i)} \\
&\leq C (2^{-k} r)^{D(1 + \rho)},
\end{split}\]
where $\rho = \log(1-\delta)/\log (\eta^D) \in (0,1)$ by \eqref{eq:eta_delta_ineq}; in the middle inequality, we used the Ahlfors condition \eqref{eq:ahlfors} and the finite overlapping property \eqref{eq:finite_overlapping} to control, for every $i \in \NN$, the number of $j \in J(i)$ such that $2^k \leq d(x_{ij}) < 2^{k+1}$, $x_{ij} \in B(x,r)$, $\eta^{1+i} d(x_{ij}) < r$ by a multiple of $( r / (\eta^{1+i} 2^k) )^D$.

Note also that, if $(x,r) \in \supp \nu_k$, then $x \in B(o,2^{k+1})$ and $r\leq 1$, so $B(x,r) \subseteq B(o,2^{k+2})$. We can then apply Lemma \ref{lem:carleson} with $q = 2/(1+\rho)$ and \eqref{eq:ahlfors} to obtain that
\[
\begin{split}
&\int_{X \times (0,\infty)} \left( \dashint_{B(x,r)} ((\opK_* f)(y))^{1/2} \wrt m(y) \right)^2 \wrt\nu_k(x,r) \\
&= \int_{X \times (0,\infty)} \left( \dashint_{B(x,r)} ((\opK_* f)(y))^{1/2} \chr_{B(o,2^{k+2})}(y) \,d m(y) \right)^2 \wrt\nu_k(x,r) \\
& \leq C 2^{-kD(1+\rho)} \| (\opK_* f)^{1/2} \chr_{B(o,2^{k+2})} \|_{q}^2 \\
& = C 2^{-kD(1+\rho)} \left(\int_{B(o,2^{k+2})} (\opK_* f)^{1/(1+\rho)} \right)^{1+\rho}  \\
& \leq C \left(\dashint_{B(o,2^{k+2})} (\opK_* f)^{1/(1+\rho)} \right)^{1+\rho} \\
& \leq C (\maxM ((\opK_* f)^{1/(1+\rho)})(o))^{1+\rho},
\end{split}
\]
which, together with \eqref{eq:first_decomp}, gives the desired estimate with $p = 1/(1+\rho)$.
\end{proof}

We can now prove the main result of this section.

\begin{proof}[Proof of Theorem \ref{thm:main}]
Let us first observe that the estimate \eqref{eq:est_phi_maxk} actually holds (with the same constants) for all $\phi \in \cutF_\gamma(o,r)$ and $r \in (0,1]$. Indeed, it is sufficient to apply Corollary \ref{cor:est_phi_maxk} to the rescaled metric $d_r$, measure $m_r$ and kernel $K_r$ given by
\[
d_r = d/r, \qquad m_r= m/r^D, \qquad K_r(t,x,y) = r^D K(rt,x,y),
\]
which satisfy the same assumptions as $d$, $m$, $K$ (with the same constants). The pointwise estimate \eqref{eq:pointwise_grand_maxk} then follows by taking the supremum for $r \in (0,1]$ and $\phi \in \cutF_\gamma(o,r)$, for arbitrary $o \in X$. This estimate, together with the boundedness of the Hardy--Littlewood maximal function $\maxM$ on $L^s(X)$ for $s \in (1,\infty]$, immediately gives \eqref{eq:lq_grand_maxk}.
\end{proof}

\subsection{Proof of the decomposition}\label{ss:proof_dec}

Here we prove the crucial Proposition \ref{prp:phidecomp}.
From now on we think of the AI $K$ and the point $o \in X$ as fixed. As in the statement of Proposition \ref{prp:phidecomp}, we define $d(x) = 1+ d(o,x)$ for all $x \in X$.

\begin{lemma}\label{lem:equiv_dist}
For all $x,y \in X$, if $d(x,y) \leq d(y)/2$, then $d(y)/2 \leq d(x) \leq 2d(y)$. Moreover, if $d(x,y) < d(y)/2$, then $d(y)/2 < d(x) < 2d(y)$.
\end{lemma}
\begin{proof}
Immediate from the triangle inequality.
\end{proof}

\begin{lemma}\label{lem:cover}
For all $a \in (0,1]$, there exists $C_a \in (0,\infty)$ such that the following hold. Let $t \in (0,1/2]$ and let $g \in L^1_\loc(X)$ be nonnegative. Then there exists a finite collection $\{x_j\}_j$ of points of $X$ such that
\begin{gather}
t d(x_j) \leq 1, \label{eq:txjbd} \\
\sum_j \chr_{B(x_j, td(x_j))}(x) \leq C_a \qquad\text{for all } x \in X, \label{eq:fin_overlap} \\
\sum_j \chr_{B(x_j, a td(x_j))}(x) \geq 1 \qquad\text{whenever } td(x) \leq 1/2, \label{eq:cover} \\
g(x_j) \leq C_a \, \dashint_{B(x_j, td(x_j))} g(y) \wrt m(y). \label{eq:average_lbd}
\end{gather}
\end{lemma}
\begin{proof}
Let $\{y_j\}_j$ be a collection of points of $X_t = \{ y \in X \tc td(y) \leq 1/2\}$ maximal with respect to the condition
\begin{equation}\label{eq:min_distance}
d(y_j,y_k) \geq a t \min\{d(y_j),d(y_k)\}/4 \qquad\text{for all distinct } j,k.
\end{equation}
Finiteness of the collection is an immediate consequence of \eqref{eq:ahlfors}. Moreover, by maximality,
\begin{equation}\label{eq:close_points}
\text{for all $x \in X_t$, there exists $j$ such that $d(x,y_j) < a t \min\{d(x),d(y_j)\}/4$.}
\end{equation}

For each $j$, choose now $x_j \in B(y_j,a td(y_j)/4)$ such that
\begin{equation}\label{eq:average_y}
g(x_j) \leq 2 \, \dashint_{B(y_j,a td(y_j)/4)} g(y) \wrt m(y).
\end{equation}
By Lemma \ref{lem:equiv_dist},
\begin{equation}\label{eq:comp_dist_xy}
d(x_j)/2 \leq d(y_j) \leq 2d(x_j);
\end{equation}
consequently \eqref{eq:txjbd} holds, and moreover $B(y_j,a td(y_j)/4) \subseteq B(x_j,td(x_j))$, so \eqref{eq:average_lbd} follows by \eqref{eq:average_y} and \eqref{eq:ahlfors}.

Similarly, for all $x \in X_t$, by \eqref{eq:close_points} there exists $j$ such that
\[
d(x,x_j) < a td(y_j)/4 + a td(y_j)/4 \leq a td(x_j),
\]
which implies \eqref{eq:cover}.

Finally, for all $x \in X$, if $x \in B(x_j,td(x_j))$, then, again by Lemma \ref{lem:equiv_dist} and \eqref{eq:comp_dist_xy},
\[
d(y_j)/4 \leq d(x) \leq 4d(y_j),
\]
whence $y_j \in B(x,3td(x))$; moreover, by \eqref{eq:min_distance} such points $y_j$ are at least at distance $a td(x)/16$ from each other, and therefore \eqref{eq:fin_overlap} follows from \eqref{eq:ahlfors}.
\end{proof}

\begin{lemma}\label{lem:summing}
Let $L \in (0,\infty)$ and $a,b \in [0,\infty)$ be such that $b \geq a$. Then there exists $C_{a,b,L} \in (0,\infty)$ such that the following hold.
Let $t \in (0,1)$ and let $\{x_j\}_j$ be a collection of points of $X$ such that
\begin{equation}\label{eq:fin_overlap_lemma}
\sup_{x \in X} \sum_j \chr_{B(x_j,td(x_j))}(x) \leq L.
\end{equation}
Then, for all $x \in X$ and $h \in [0,\infty)$,
\begin{multline}\label{eq:sumest_main}
\sum_{j \tc d(x_j,x) \geq h td(x_j)} d(x_j)^{-D-a} (1+d(x_j,x)/(td(x_j)))^{-D-b} \\
\leq C_{a,b,L}   \, d(x)^{-D-a} \max\{t^b,(1+h)^{-b}\}.
\end{multline}
In addition, if $td(x) \geq 2$, then
\begin{equation}\label{eq:sumest_tail}
\sum_{j \tc td(x_j) \leq 1} d(x_j)^{-D-a} (1+d(x_j,x)/(td(x_j)))^{-D-b} 
\leq C_{a,b,L}   \, d(x)^{-D-b} t^a.
\end{equation}
\end{lemma}
\begin{proof}
Note that, by the triangle inequality,
\[\begin{split}
(1+d(x_j,x)/(td(x_j)))^{-D-b} &\leq C \inf_{y \in B(x_j,td(x_j))} (1+d(y,x)/(td(x_j)))^{-D-b} \\
&\leq C \dashint_{B(x_j,td(x_j))} (1+d(y,x)/(td(x_j)))^{-D-b} \wrt m(y),
\end{split}\]
where $C$ may depend on $b$.
Let $k \in \NN$. If $2^k \leq d(x_j) < 2^{k+1}$, then, by \eqref{eq:ahlfors},
\begin{multline*}
d(x_j)^{-D-a} (1+d(x_j,x)/(td(x_j)))^{-D-b} \\
\leq C 2^{-k(D+a)} (t2^k)^{-D} \int_{B(x_j,td(x_j))} (1+d(y,x)/(t2^k))^{-D-b} \wrt m(y).
\end{multline*}
Hence, if $\tilde h = (h-1)_+$, then, by \eqref{eq:fin_overlap_lemma} and \eqref{eq:ahlfors},
\begin{equation}\label{eq:est_main}
\begin{split}
&\sum_{\substack{j \tc d(x_j,x) \geq h td(x_j) \\ 2^k \leq d(x_j) < 2^{k+1}}} d(x_j)^{-D-a} (1+d(x_j,x)/(td(x_j)))^{-D-b} \\
&\leq C 2^{-k(D+a)} (t2^k)^{-D} \int_X (1+d(y,x)/(t2^k))^{-D-b} \chr_{[\tilde h,\infty)}(d(y,x)/t2^k) \wrt m(y) \\
&\leq C 2^{-k(D+a)} \sum_{\ell \geq 0 \tc 2^\ell \geq \tilde h} 2^{D\ell} 2^{-\ell(D+b)} \\
&\leq C 2^{-k(D+a)} (1+h)^{-b},
\end{split}
\end{equation}
where $C$ may depend on $b$ and $L$.

Let us also remark that, by \eqref{eq:fin_overlap_lemma} and \eqref{eq:ahlfors}, the number of $j$ such that $2^k \leq d(x_j) < 2^{k+1}$ is bounded by a multiple of $t^{-D}$. Hence, if $d(x) < 2^{k-1}$, then $d(x_j,x) > d(x_j)/2$ by Lemma \ref{lem:equiv_dist} and
\begin{equation}\label{eq:est_xsmall}
\sum_{\substack{j \tc 2^k \leq d(x_j) < 2^{k+1}}} d(x_j)^{-D-a} (1+d(x_j,x)/(td(x_j)))^{-D-b} 
\leq C 2^{-k(D+a)} t^{b};
\end{equation}
if instead $2^{k+2} < d(x)$, then $d(x_j,x) > d(x)/2$ and
\begin{equation}\label{eq:est_xlarge}
\sum_{\substack{j \tc 2^k \leq d(x_j) < 2^{k+1}}} d(x_j)^{-D-a} (1+d(x_j,x)/(td(x_j)))^{-D-b} 
\leq C 2^{k(b-a)} t^{b} d(x)^{-D-b}.
\end{equation}

Summing over $k \in \NN$ by exploiting the estimates \eqref{eq:est_main}, \eqref{eq:est_xsmall} and \eqref{eq:est_xlarge} immediately gives \eqref{eq:sumest_main}. On the other hand, if $td(x) \geq 2$ and $td(x_j) \leq 1$, then $2d(x_j) \leq d(x)$ and $d(x,x_j) \geq d(x)/2$ by Lemma \ref{lem:equiv_dist}, so \eqref{eq:est_xlarge} applies; however in this case the sum is restricted to $2^{k} \leq t^{-1}$, which leads to \eqref{eq:sumest_tail}.
\end{proof}

\begin{proof}[Proof of Proposition \ref{prp:phidecomp}]
Let $\delta,\eta \in (0,1)$ and $\kappa \in (1,\infty)$ be constants to be fixed later. We will see throughout the proof what constraints are needed on $\delta,\eta,\kappa$.

Assume that
\begin{equation}\label{eq:eta_firstassumption}
\eta \leq 1/2.
\end{equation}
Let $c_2$ be the constant in \eqref{eq:nonvanishing}.
For all $i \in \NN$, by applying Lemma \ref{lem:cover} with $t = \eta^{1+i}$, $a = c_2$, and $g = (\opK_* f)^{1/2}$, we construct a finite family $\{x_{ij}\}_{j \in J(i)}$ of points of $X$ satisfying
\begin{gather}
\sup_{x \in X} \sum_{j \in J(i)} \chr_{ B(x_{ij},\eta^{1+i} d(x_{ij}))}(x) \leq L, \qquad\text{for all } x \in X, \label{eq:fin_overlap_ij} \\
\sum_{j \in J(i)} \chi_{ B(x_{ij},c_2\eta^{1+i} d(x_{ij}))}(x) \geq 1 \qquad\text{whenever } \eta^{1+i} d(x) \leq 1/2,  \label{eq:covering_enough} \\
\eta^{i+1} d(x_{ij}) \leq 1, \label{eq:max_dist} \\
(\opK_* f(x_{ij}))^{1/2} \leq L \, \dashint_{B(x_{ij},\eta^{1+i} d(x_{ij}))} (\opK_* f(y))^{1/2} \wrt m(y), \label{eq:avg_bd} 
\end{gather}
for all $j \in J(i)$, where $L \in (0,\infty)$ is independent of $f$, $\eta$ and $i$.
In particular, \eqref{eq:smalltimes}, \eqref{eq:average_max} and \eqref{eq:finite_overlapping} are certainly satisfied.

\bigskip

Let $\phi \in \cutF_\gamma(o,1)$. Up to rescaling it is not restrictive to assume that
\begin{equation}\label{eq:phi_uniformbound}
\|\phi\|_\infty \leq 2^{-D-\ga/2}.
\end{equation}
Let us now define recursively, for all $i \in \NN$, the function $\phi_i : X \to \RR$ by setting $\phi_0 = \phi$ and $\phi_{i+1} = \phi_i - w_i$, where
\[
w_i(x) = \kappa \delta (1-\delta)^i \sum_{j \in J(i)} \epsilon_{ij} d(x_{ij})^{-\ga/2} \eta^{D(1+i)} K(\eta^{1+i} d(x_{ij}), x_{ij}, x)
\]
and $\epsilon_{ij} = \sgn \phi_i(x_{ij})$.
We now want to prove, for all $i \in \NN$, that
\begin{equation}\label{eq:phis_bound1} 
|\phi_i(x)| \leq (1-\delta)^i d(x)^{-D-\ga/2} \qquad\text{for all } x \in X.
\end{equation}
Clearly this implies that $\phi_i \to 0$ locally uniformly as $i\to \infty$, and consequently the representation \eqref{eq:phidecomp} holds, provided we relabel $\kappa\delta$ as $\kappa$.

We will prove \eqref{eq:phis_bound1} by induction on $i$. 
Note that, because of \eqref{eq:phi_uniformbound}, the estimate \eqref{eq:phis_bound1} trivially holds for $i=0$. 
Before entering the proof of the induction step, we discuss a number of useful estimates.
\bigskip

Let us first obtain a few ``a priori'' estimates for the functions $w_i$ (that do not depend on the choices of the signs $\epsilon_{ij}$).
Let $i \in \NN$. By \eqref{eq:Kupper}, \eqref{eq:fin_overlap_ij} and Lemma \ref{lem:summing}, for all $x \in X$,
\[\begin{split}
|w_i(x)| 
&\leq \kappa \delta (1-\delta)^i \sum_{j \in J(i)} d(x_{ij})^{-\ga/2} \eta^{D(1+i)} K(\eta^{1+i} d(x_{ij}), x_{ij}, x) \\
&\leq \kappa \delta (1-\delta)^i \sum_{j \in J(i)} d(x_{ij})^{-D-\ga/2} (1+d(x_{ij},x)/(\eta^{1+i} d(x_{ij})))^{-D-\ga} \\
&\leq C_{\ga/2,\ga,L} \, \kappa \delta (1-\delta)^i d(x)^{-D-\ga/2};
\end{split}\]
moreover, if $\eta^{i+1} d(x) \geq 2$, then 
\[
|w_i(x)| \leq C_{\ga/2,\ga,L} \, \kappa \delta (1-\delta)^i d(x)^{-D-\ga} (\eta^{1+i})^{\ga/2}.
\]
Hence, if we assume that
\begin{equation}\label{cond1}
C_{\ga/2,\ga,L} \, \kappa \delta \leq 1/4,
\end{equation}
then
\begin{equation}\label{eq:wi_est}
|w_i(x)| \leq \begin{cases}
\frac{1}{4} (1-\delta)^i d(x)^{-D-\ga/2} &\text{for all } x \in X, \\
\frac{1}{4} \eta^{\ga/2} ((1-\delta) \eta^{\ga/2})^i d(x)^{-D-\ga} &\text{if } \eta^{1+i} d(x) \geq 2.
\end{cases}
\end{equation}

Similarly, for all $x,y \in X$ such that $d(x,y) \leq \eta^{1+i} d(x)/4$, by the triangle inequality we deduce that, for all $j \in J(i)$,
\[
d(x,y) \leq (\eta^{1+i} d(x_{ij}) + d(x,x_{ij}))/4,
\]
hence, by \eqref{eq:Klip} and Lemma \ref{lem:summing},
\[\begin{split}
&|w_i(x) - w_i(y)| \\
&\leq \kappa \delta (1-\delta)^i \sum_{j \in J(i)} d(x_{ij})^{-\ga/2} \eta^{D(1+i)} |K(\eta^{1+i} d(x_{ij}), x_{ij}, x) - K(\eta^{1+i} d(x_{ij}), x_{ij}, y)| \\
&\leq \kappa \delta (1-\delta)^i (d(x,y)/\eta^{1+i})^\ga \sum_{j \in J(i)} d(x_{ij})^{-D-3\ga/2}  (1+d(x_{ij}, x)/(\eta^{1+i} d(x_{ij})))^{-D-2\ga} \\
&\leq C_{3\ga/2,2\ga,L} \frac{\kappa \delta}{1-\delta} ((1-\delta)/\eta^\ga)^{i+1} d(x,y)^\ga d(x)^{-D-3\ga/2}.
\end{split}\]
Hence, if we assume that
\begin{equation}\label{cond2}
C_{3\ga/2,2\ga,L} \, \kappa \delta \leq 1/4, \quad 1-\delta \geq 3/4,
\end{equation}
then, provided $d(x,y) \leq \eta^{1+i} d(x)/4$,
\begin{equation}\label{eq:wi_lip}
|w_i(x) - w_i(y)| \leq 
\frac{1}{3}  ((1-\delta)/\eta^\ga)^{i+1} d(x,y)^\ga d(x)^{-D-3\ga/2} .
\end{equation}

\bigskip

By summing the estimates \eqref{eq:wi_est}, we can immediately obtain the validity of a stronger estimate than \eqref{eq:phis_bound1} for $x$ in a suitable region (depending on $i$).

Indeed, note that, if $d(x) \geq 2$, then $d(x,o) \geq 1$ and therefore $\phi(x) = 0$. Consequently, if we assume that
\begin{equation}\label{cond3}
 \eta^{\ga/2} \leq 1/2,
\end{equation}
then, for all $i \in \NN$, by \eqref{eq:wi_est},
\begin{equation}\label{eq:phi_far}
\begin{split}
|\phi_i(x)| &\leq \sum_{s=0}^{i-1} |w_s(x)| \\
&\leq \frac{1}{4} \eta^{\ga/2} d(x)^{-D-\ga} \sum_{s=0}^{i-1}((1-\delta) \eta^{\ga/2})^s \\
&\leq \frac{1}{2} \eta^{\ga/2} d(x)^{-D-\ga} \\
&\leq \frac{1}{2^{1+\ga/2}} (\eta^{\ga/2})^{1+i} d(x)^{-D-\ga/2} \\
&\leq \frac{1}{4} (1-\delta)^i d(x)^{-D-\ga/2},
\end{split}
\end{equation}
where we used that $\eta^{\ga/2} \leq 1-\delta$ by \eqref{cond2} and \eqref{cond3}.
This shows that \eqref{eq:phis_bound1} is trivially satisfied whenever $\eta^i d(x) \geq 2$.

\bigskip

By summing the estimates \eqref{eq:wi_lip}, we can also derive, for all $i \in \NN$, an useful estimate for the difference of the values of $\phi_i$ at different points.

Observe first that, since $\phi \in \cutF_\gamma(o,1)$, for all $x,y\in X$ such that $d(x,y) \leq d(x)/2$, by Lemma \ref{lem:equiv_dist} and the support condition, the difference $|\phi(x) - \phi(y)|$ vanishes unless $d(x) \leq 4$, so
\[
|\phi(x) - \phi(y)| \leq 4^{D+3\ga/2} d(x,y)^\ga d(x)^{-D-3\ga/2}.
\]
Hence, by \eqref{eq:wi_lip}, if $d(x,y) \leq \eta^i d(x)/4$, then
\[\begin{split}
|\phi_i(x) - \phi_i(y)| &\leq |\phi(x) - \phi(y)| + \frac{1}{3} d(x,y)^\ga d(x)^{-D-3\ga/2} \sum_{s=0}^{i-1} ((1-\delta)/\eta^\ga)^{s+1} \\
&\leq (4^{D+3\ga/2} + 2/3) (1-\delta)^i d(x)^{-D-\ga/2} (d(x,y)/(\eta^i d(x)))^\ga ,
\end{split}\]
provided
\begin{equation}\label{cond5}
 1-\delta \geq 2 \eta^\ga.
\end{equation}
Consequently, if we define
\[
\sigma = \min\{1/4, (2(4^{D+3\ga/2} + 2/3))^{-1/\ga}\},
\]
then, for all $x,y\in X$ such that $d(x,y) \leq \sigma \eta^i d(x)$,
\begin{equation}\label{eq:holderbdphii}
|\phi_i(x) - \phi_i(y)| \leq \frac{1}{2} (1-\delta)^i d(x)^{-D-\ga/2}
\end{equation}

\bigskip

We now proceed with the proof of the inductive step; i.e., for a given $i \in \NN$, we assume the validity of \eqref{eq:phis_bound1} for $\phi_i$ and prove the same estimate for $\phi_{i+1}$.

Let $x \in X$. Consider first the case where
\[
|\phi_i(x)| \leq \frac{1}{2} (1-\delta)^i d(x)^{-D-\ga/2},
\]
In this case, since $3/4 \leq 1-\delta$ by \eqref{cond2},
from \eqref{eq:wi_est} we deduce that
\[\begin{split}
|\phi_{i+1}(x)| &\leq |\phi_i(x)| + |w_i(x)| \\
&\leq (3/4) (1-\delta)^i d(x)^{-D-\ga/2}  \\
&\leq (1-\delta)^{i+1} d(x)^{-D-\ga/2},
\end{split}\]
and we are done.

Hence, to prove \eqref{eq:phis_bound1} for $\phi_{i+1}$, it remains to consider the case where
\begin{equation}\label{eq:regime}
\frac{1}{2} (1-\delta)^i d(x)^{-D-\ga/2} < |\phi_i(x)| \leq (1-\delta)^i d(x)^{-D-\ga/2}.
\end{equation}
On the other hand, if we assume that $\eta^{i+1} d(x) \geq 1/2$, then
\[
\eta^i d(x) \geq \eta^{-1}/2 \geq 2,
\]
provided
\begin{equation}\label{cond7}
 \eta \leq 1/4,
\end{equation}
whence, by \eqref{eq:phi_far}, $|\phi_i(x)| \leq \frac{1}{4} (1-\delta)^i d(x)^{-D-\ga/2}$. So, being in the regime \eqref{eq:regime} implies that
\begin{equation}\label{eq:dist_bnd}
\eta^{i+1} d(x) < 1/2.
\end{equation}

From \eqref{eq:holderbdphii} and \eqref{eq:regime} we deduce that, for all $y \in X$, if $d(x,y) \leq \sigma \eta^i d(x)$, then
\begin{equation}\label{eq:equalsign}
\sgn \phi_i(y) = \sgn \phi_i(x).
\end{equation}
Set $\epsilon = \sgn \phi_i(x)$; then
\[\begin{split}
\epsilon w_i(x)
&= \kappa \delta (1-\delta)^i \sum_{j \tc \epsilon \phi_i(x_{ij}) > 0} d(x_{ij})^{-\ga/2} \eta^{D(1+i)} K(\eta^{1+i} d(x_{ij}), x_{ij}, x) \\
&\quad - \kappa \delta (1-\delta)^i \sum_{j \tc \epsilon \phi_i(x_{ij}) < 0} d(x_{ij})^{-\ga/2} \eta^{D(1+i)} K(\eta^{1+i} d(x_{ij}), x_{ij}, x) \\
&\geq \kappa \delta (1-\delta)^i \sum_{j \in J(i)} d(x_{ij})^{-\ga/2} \eta^{D(1+i)} K(\eta^{1+i} d(x_{ij}), x_{ij}, x) \\
&\quad - 2\kappa \delta (1-\delta)^i \sum_{j \tc \epsilon \phi_i(x_{ij}) \leq 0} d(x_{ij})^{-\ga/2} \eta^{D(1+i)} K(\eta^{1+i} d(x_{ij}), x_{ij}, x).
\end{split}\]
Note now that, by \eqref{eq:equalsign}, \eqref{eq:Kupper} and Lemma \ref{lem:summing},
\[\begin{split}
&\sum_{j \tc \epsilon \phi_i(x_{ij}) \leq 0} d(x_{ij})^{-\ga/2} \eta^{D(1+i)} K(\eta^{1+i} d(x_{ij}), x_{ij}, x) \\
&\leq \sum_{j \tc d(x,x_{ij}) > \sigma \eta^i d(x)} d(x_{ij})^{-D-\ga/2} (1+d(x_{ij},x)/(\eta^{1+i} d(x_{ij})))^{-D-\ga}  \\
&\leq C_{\ga/2,\ga,L}\, \sigma^{-\ga} \eta^{\ga} d(x)^{-D-\ga/2}.
\end{split}\]
On the other hand, since $K$ is nonnegative, by \eqref{eq:nonvanishing}, Lemma \ref{lem:equiv_dist}, \eqref{eq:dist_bnd}, and \eqref{eq:covering_enough}, 
\[\begin{split}
&\sum_{j \in J(i)} d(x_{ij})^{-\ga/2} \eta^{D(1+i)} K(\eta^{1+i} d(x_{ij}), x_{ij}, x) \\
&\geq \sum_{j \in J(i)} d(x_{ij})^{-\ga/2} \eta^{D(1+i)} K(\eta^{1+i} d(x_{ij}), x_{ij}, x) \chr_{B(x_{ij},c_2\eta^{1+i} d(x_{ij}))}(x) \\
&\geq c_1 2^{-D-\ga/2} \, d(x)^{-D-\ga/2}.
\end{split}\]
Hence, if we assume that
\begin{equation}\label{cond8}
\eta^\ga \leq \frac{c_1 2^{-D-\ga/2}}{4 C_{\ga/2,\ga,L} \, \sigma^{-\ga}}, \qquad \kappa = (c_1 2^{-1-D-\ga/2})^{-1}
\end{equation}
then
\[
\epsilon w_i(x)
\geq \delta (1-\delta)^i d(x)^{-D-\ga/2}
\]
and consequently, by \eqref{eq:wi_est} and \eqref{eq:regime},
\[\begin{split}
|\phi_{i+1}(x)| &= \epsilon \phi_{i+1}(x) \\
&\leq (1-\delta)^i d(x)^{-D-\ga/2} - \delta (1-\delta)^i d(x)^{-D-\ga/2} \\
&= (1-\delta)^{i+1} d(x)^{-D-\ga/2},
\end{split}\]
which concludes the proof of the inductive step.

By looking at all the above conditions, one sees that they are satisfied if we first fix the value of $\kappa$ as in \eqref{cond8}, then we choose $\delta$ sufficiently small that \eqref{cond1} and \eqref{cond2} are satisfied, and finally we choose the value of $\eta$ sufficiently small that all conditions \eqref{eq:eta_firstassumption}, \eqref{cond3}, \eqref{cond5}, \eqref{cond7}, \eqref{cond8} and \eqref{eq:eta_delta_ineq} are satisfied.
\end{proof}

\section{Maximal characterisation of $\ghu(M)$ via approximations of the identity}\label{s:manifold_maximal}

We now return to the setting of a Riemannian manifold $M$ satisfying the assumptions \ref{en:ass_inj}-\ref{en:ass_ric} of 
Section~\ref{s: Background material}.  
In this section we shall prove a maximal characterisation of the atomic local Hardy space $\ghu(M)$ 
in terms of a local maximal function associated to a single ``approximation of the identity", 
in the sense defined below.

\begin{definition}\label{def: LAI}
Let $\gamma \in (0,1]$ and $\la > 0$. A \emph{$\lambda$-local approximation of the identity} ($\la$-LAI in the sequel) of exponent $\gamma$ on $M$ is a measurable function 
$K:(0,1]\times M\times M\to [0,\infty)$ for which there exist positive constants
$C_1$, $C_2$ and $C_3$ such that for every $t\in (0,1]$, $x,y,z\in M$, with
\begin{equation}\label{eq:LAI_range}
4d(y,z)\leq t+d(x,y),
\end{equation}
the following hold:
\begin{enumerate}[label=(\roman*)]
\item\label{en:lai_supp} $K(t,x,y)=0$ if $d(x,y)>\la$;
\item\label{en:lai_upper} $K(t,x,y)\leq C_1\, t^{-n} (1+d(x,y)/t)^{-n-\gamma}$;
\item\label{en:lai_nondeg} $K(t,x,x)\geq C_2 \, t^{-n}$;
\item\label{en:lai_lip} $|K(t,x,y)-K(t,x,z)|
	\leq C_3 \, t^{-n}\, (d(y,z)/t)^{\gamma} (1+d(x,y)/t)^{-n-2\gamma}$.
\end{enumerate}
We denote by $\aiV_\gamma$ the collection of all LAI of exponent $\gamma$ on $M$.
\end{definition}

\begin{definition}\label{def: AI}
Let $\gamma \in (0,1]$. An \emph{approximation of the identity on $M$} (AI on $M$ in the sequel) of exponent $\gamma$ is a measurable function 
$K:(0,1]\times M\times M\rightarrow \CC$ which can be written as $K = K_1+K_2$, where $K_1$ is in $\aiV_\gamma$ and
\begin{equation}\label{eq:ai_l1bound}
\esssup_{y \in M} \int_M \sup_{t \in (0,1]} |K_2(t,x,y)| \wrt\mu(x) < \infty.
\end{equation}
We denote by $\taiV_\gamma$ the collection of all AI of exponent $\gamma$ on $M$. 
\end{definition}

\begin{remark}
Definition \ref{def: LAI} is analogous to Definition \ref{def: AI Ahlfors}, but includes the additional constraint for $K(t,\cdot,\cdot)$ to be supported in a $t$-independent neighbourhood of the diagonal. Definition \ref{def: AI} provides a relaxation of the support constraint, which is very convenient in applications. As a matter or fact, in the case $M$ is globally $n$-Ahlfors, one can show that any kernel $K$ satisfying the bounds \ref{en:lai_upper} to \ref{en:lai_lip} of Definition \ref{def: LAI} is actually an AI in the sense of Definition \ref{def: AI}. In these respects, Definition \ref{def: AI} can be considered as an appropriate extension of Definition \ref{def: AI Ahlfors} that applies also to spaces that are locally, but not globally Ahlfors.
\end{remark}

In this section, the exponent $\gamma \in (0,1]$ will be thought of as fixed, and we will simply write ``$\lambda$-LAI'' in place of ``$\lambda$-LAI of exponent $\gamma$''. We will also write $\aiV$ and $\taiV$ in place of $\aiV_\gamma$ and $\taiV_\gamma$.

\begin{remark}\label{rem:LAI_diff}
In case $K(t,x,y)$ is continuously differentiable in $y$, condition \ref{en:lai_lip} in Definition \ref{def: LAI} is implied by the differential condition
\begin{enumerate}[label=(\roman*'),start=4]
\item $\mod{\nabla_y K(t,x,y)}
	\leq C_4 \, t^{-n-1}  (1+d(x,y)/t)^{-n-1-\gamma}$.
\end{enumerate}
Indeed, by the fundamental theorem of calculus,
\[
|K(t,x,y) - K(t,x,z)| \leq d(y,z) \sup_{w \in \Gamma(y,z)} |\nabla_y K(t,x,w)|
\]
where $\Gamma(y,z)$ is a length-minimising arc on $M$ joining $y$ to $z$; since $d(x,w) \geq d(x,y)-d(y,w) \geq d(x,y) - d(y,z)$ for all $w \in \Gamma(y,z)$, in the range \eqref{eq:LAI_range} we deduce that $1+d(x,w)/t \geq (3/4) (1+d(x,y)/t)$ and
\[\begin{split}
|K(t,x,y) - K(t,x,z)| 
&\leq C_4 (4/3)^{n+1+\gamma} t^{-n} (d(y,z)/t) (1+d(x,y)/t)^{-n-1-\gamma} \\
&\leq C_3 t^{-n} (d(y,z)/t)^{\gamma} (1+d(x,y)/t)^{-n-2\gamma}
\end{split}\]
where $C_3 = (4/3)^{n+1+\gamma} (1/4)^{1-\gamma} C_4$. 
\end{remark}

Clearly the restriction to the range \eqref{eq:LAI_range} in Definition \ref{def: LAI} is only relevant for condition \ref{en:lai_lip}. As a matter of fact, the constant $4$ in the range \eqref{eq:LAI_range} could be replaced with any other constant greater than $1$ without changing the class of $\lambda$-LAI, as shown by the following lemma and its proof.

\begin{lemma}\label{lem:lai_lip_range}
Let $\kappa > 0$. Let $K : (0,1] \times M \times M \to [0,\infty)$ satisfy the conditions \ref{en:lai_upper} and \ref{en:lai_lip} of Definition \ref{def: LAI} for some constants $C_1,C_3>0$ and all $t \in (0,1]$ and $x,y,z \in M$ in the range
\[
\kappa d(y,z)\leq t+\min\{ d(x,y), d(x,z)\}.
\]
Then $K$ also satisfies condition \ref{en:lai_lip} in the range \eqref{eq:LAI_range},
with a constant $C_3'$ (in place of $C_3$) only depending on $\kappa,C_1,C_3$.
\end{lemma}
\begin{proof}
We only need to check condition \ref{en:lai_lip} in the range
\[
(t+\min\{ d(x,y), d(x,z)\})/\kappa < d(y,z) \leq (t+d(x,y))/4.
\]
However in this range, by the triangle inequality,
\[
 1+\frac{d(x,z)}{t} \geq \frac{3}{4} \left(1+\frac{d(x,y)}{t}\right), \qquad \frac{d(y,z)}{t} > \frac{3}{4\kappa}\left(1+\frac{d(x,y)}{t}\right)
\]
so condition \ref{en:lai_lip} follows from \ref{en:lai_upper} with $C_3' = \max\{C_3, C_1 (1+(4/3)^{n+\gamma}) (4\kappa/3)^\gamma\}$.
\end{proof}

A simple example of $\la$-LAI on $M$ is
\begin{equation}\label{eq:LAI_example}
K(t,x,y) = t^{-n} \psi(d(x,y)/t),
\end{equation}
where $\psi : [0,\infty) \to \RR$ is $\gamma$-H\"older, $\psi(0) > 0$ and $\supp \psi \subseteq [0,\la]$. Moreover, the following lemma shows how one can construct new LAI by localising existing ones via suitable H\"older cutoffs.

\begin{lemma}\label{lem:lai_cutoff}
Let $K : (0,1] \times M \times M \to [0,\infty)$ satisfy \ref{en:lai_upper} and \ref{en:lai_lip} of Definition \ref{def: LAI} for some constants $C_1,C_3>0$. Let $\Phi : M \times M \to [0,\infty)$ be such that
\[
|\Phi(x,y)| \leq L, \qquad |\Phi(x,y) - \Phi(x,z)| \leq L d(y,z)^\gamma
\]
for some $L > 0$ and all $x,y,z \in M$. Let $\la>0$ and assume further that either $K$ satisfies \ref{en:lai_supp} of Definition \ref{def: LAI}, or
\[
\Phi(x,y) = 0 \text{ if } d(x,y) > \lambda.
\]
Then $K' : (0,1] \times M \times M \to [0,\infty)$ defined by
\[
K'(t,x,y) = \Phi(x,y) K(t,x,y)
\]
satisfies \ref{en:lai_supp}, \ref{en:lai_upper} and \ref{en:lai_lip} of Definition \ref{def: LAI}, with constants $C_1',C_3'$ (in place of $C_1,C_3$) that only depend on $\la,C_1,C_3,L$.
\end{lemma}
\begin{proof}
Under our assumptions, it is immediately seen that $K'$ satisfies \ref{en:lai_supp} and \ref{en:lai_upper} of Definition \ref{def: LAI}, by taking $C_1' = C_1 L$. As for \ref{en:lai_lip}, note that, if $x,y,z \in M$  and $t \in (0,1]$ satisfy $4d(y,z) \leq t+d(x,y)$, then
\begin{equation}\label{eq:cut_lip}
\begin{split}
&|K'(t,x,y) - K'(t,x,z)| \\
&\leq |\Phi(x,y)-\Phi(x,z)| K(t,x,y) + \Phi(x,z) |K(t,x,y)-K(t,x,z)| \\
&\leq C_1 L d(y,z)^\gamma t^{-n} (1+d(x,y)/t)^{-n-1} + C_3 L t^{-n} (d(y,z)/t)^{\gamma} (1+d(x,y)/t)^{-n-1-\gamma}.
\end{split}
\end{equation}
Now, if $d(x,y) \leq 1+2\lambda$, then 
\[
t(1+d(x,y)/t) = t +d(x,y) \leq 2+2\lambda
\]
and from \eqref{eq:cut_lip} we deduce that
\begin{equation}\label{eq:cut_lip2}
|K'(t,x,y) - K'(t,x,z)| \leq L (C_1(2+2\lambda)^\gamma+ C_3) t^{-n} (d(y,z)/t)^\gamma (1+d(x,y)/t)^{-n-1-\gamma}.
\end{equation}
If instead $d(x,y) > 1+2\lambda$, then $d(x,z) \geq d(x,y) - d(y,z) \geq d(x,y) - (t+d(x,y)/4) \geq (3d(x,y)-1)/4 > (1+3\lambda)/2 > \lambda$, so the left-hand side of \eqref{eq:cut_lip} vanishes and \eqref{eq:cut_lip} is trivially true. Consequently $K'$ satisfies \ref{en:lai_lip} of Definition \ref{def: LAI} by taking $C_3' = L(C_1(2+2\lambda)+ C_3)$.
\end{proof}

We now show that, in the decomposition of any given AI, the LAI part can be chosen so to be supported arbitrarily close to the diagonal.

\begin{lemma}\label{lem:ai_decomp}
Let $K$ be an AI on $M$. Then, for all $\la>0$, there exists a decomposition $K = K_1+K_2$ where $K_1$ is a $\lambda$-LAI on $M$ and $K_2$ satisfies \eqref{eq:ai_l1bound}.
\end{lemma}
\begin{proof}
By definition, we can decompose $K = K_1' + K_2'$, where $K_1'$ is an $r$-LAI for some $r > 0$ and $K_2'$ satisfies \eqref{eq:ai_l1bound}. If $\la \geq r$, there is nothing to prove (we can take $K_1 = K_1'$ and $K_2 = K_2'$). Suppose instead that $\la < r$, and define
\[
K_1(t,x,y) = \phi(d(x,y)) K_1'(x,y),
\]
where $\phi : \RR \to [0,1]$ is smooth, $\supp \phi \subseteq [-\lambda,\lambda]$ and $\phi|_{[-\lambda/2,\lambda/2]} \equiv 1$. By Lemma \ref{lem:lai_cutoff}, $K_1$ is a $\lambda$-LAI, so it only remains to check that
\[
K_2''(t,x,y) =  (1-\phi(d(x,y))) K_1'(x,y)
\]
satisfies \eqref{eq:ai_l1bound}.

On the other hand, $K_2''(t,x,y)$ vanishes whenever $d(x,y) < \lambda/2$ or $d(x,y) > r$. If instead $\lambda/2 \leq d(x,y) \leq r$,
\[\begin{split}
K_2''(t,x,y) &\leq C t^{-n} (1+d(x,y)/t)^{-n-\gamma} \\
&\leq C t^\gamma
\end{split}\]
where the last constant $C$ may depend on $\lambda$ and $r$. Since $K_2''$ is nonnegative, this proves that $\sup_{t \in (0,1]} |K_2''(t,x,y)| \leq C \chr_{[\lambda/2,r]}(d(x,y))$, and from the uniform local $n$-Ahlfors property we immediately deduce that $K_2''$ satisfies \eqref{eq:ai_l1bound}.
\end{proof}

To each measurable $K : (0,1] \times M \times M \to \CC$ and $t$ in $(0,1]$,
we associate
the integral operator $\opK_t$ and the maximal operator $\opK_*$ by the rules
\[
\opK_tf(x)
=  \int_M K(t,x,y)\, f(y)  \wrt\mu(y)
\quad\hbox{and}\quad
\opK_*f(x)
=  \sup_{t\in (0,1]}\, |\opK_t f(x) |
\]
respectively. 

Let $\maxM_R$ denotes the centred local Hardy--Littlewood maximal operator at scale $R>0$, defined by 
\[
\maxM_R f (x)
\defeq \sup_{0<r\leq R} \dashint_{B_r(x)} \mod{f(y)} \wrt\mu(y).  
\]
It is well known (see, e.g., \cite[Proposition 2.2]{LSW}) that, under our assumptions on $M$, $\maxM_R$ is of weak type (1,1) and bounded on $L^p(M)$ for all $p \in (1,\infty]$.

\begin{lemma}\label{lem:HLbd}
Let $\la, C_1> 0$. Then there exists $C>0$ such that, if $K : (0,1] \times M \times M \to [0,\infty)$ satisfies \ref{en:lai_supp} and \ref{en:lai_upper} of Definition \ref{def: LAI}, then 
\begin{equation} \label{f: control with local HL}
\opK_* f 
\leq C\, \maxM_\lambda f
\quant f \in L_{\loc}^1(M).
\end{equation}
\end{lemma}
\begin{proof}
By our assumptions on $K$,
\[
\begin{split}
\mod{\opK_t f(x)}
	& \leq  C_1\, \int_{B_\la(x)} \frac{t}{(t+d(x,y))^{n+\gamma}} \, \mod{f(y)} \wrt \mu(y) \\
	& \leq  C_1\, \sum_{j\geq 0} \, \int_{B_{2^{-j} \la}(x)\setminus B_{2^{-j-1} \la}(x)}
	      \frac{t}{(t+2^{-j-1}\la)^{n+\gamma}} \, \mod{f(y)} \wrt \mu(y) \\
	& \leq  C\, \sum_{j\geq 0} \, \frac{2^jt}{(2^{j}t+1)^{n+\gamma}} \,\, \maxM_Rf(x),  
\end{split}
\]
where we used the uniform local $n$-Ahlfors property of $M$.
It is straightforward to check that the series above is uniformly bounded with respect to~$t$, and the
required estimate follows.
\end{proof}

\begin{proposition}\label{p: limmaxop}
For each $K$ in $\aiV$ the maximal operator $\opK_*$ is of weak type (1,1), bounded on $L^p(M)$ for $p \in (1,\infty]$, and bounded from $\ghu(M)$ to $L^1(M)$. 
\end{proposition}
\begin{proof}

Assume that $K$ is a $\la$-LAI. By Lemma \ref{lem:HLbd} and the boundedness properties of the local Hardy--Littlewood maximal operator, we immediately deduce that $\opK_*$ is of weak-type $(1,1)$ and bounded on $L^p(M)$ for $p \in (1,\infty]$. Hence, in light of Lemma \ref{lemma:sublinear_hardy}, to conclude that $\opK_*$ is bounded from $\ghu(M)$ to $L^1(M)$, it is enough to show that $\opK_*$ is uniformly bounded on $2$-atoms at scale $\la$.

Let $a$ be a standard $2$-atom supported in a ball $B$ with centre $c_B$ and radius $r_B\leq \la$. 
Then 
\begin{equation}\label{eq:local_part}
\norm{\opK_*a}{L^1(5B)}
\leq C\mu(5 B)^{1/2}\,\norm{a}{2} 
\leq C, 
\end{equation}
by the $L^2$-boundedness of $\opK_*$ and the local doubling property.
Next, observe that $d(x,y)\geq 4 d(y,c_B)$ for every $x$ in $M\setminus (5B)$ and $y$ in $B$. 
By using the cancellation condition of the atom and Definition~\ref{def: LAI}~\ref{en:lai_lip}, we see that
\[
\begin{split}
	\mod{\opK_ta(x)}
	& \leq \int_M \mod{K(t,x,y)-K(t,x,c_B)} \,  \mod{a(y)} \wrt \mu(y)\\
	& \leq C \, t^{-n}\int_{B} \, \left(\frac{d(y,c_B)}{t}\right)^\gamma \, \left(1+\frac{d(x,y)}{t} \right)^{-n-2\gamma}|\, a(y)|\wrt \mu(y)\\
	& \leq C \, \frac{(tr_B)^\gamma}{ (t+d(x,c_B))^{n+2\gamma}}.  
\end{split}
\]
By optimizing with respect to $t$ in $(0,1]$, we find that 
\[
\sup_{t\in (0,1]} |\opK_ta(x)|
\leq C\, \frac{r_B^\gamma}{d(x,c_B)^{n+\gamma}}\,.
\]
Notice that, since $K$ is a $\lambda$-LAI, $\opK_*a$ is supported in $B_{2\lambda}(c_B)$.  
Therefore
\[\begin{split}
	\norm{\opK_*a}{L^1((5B)^c)}
	&\leq C \, r_B^\gamma \int_{B(c_B,2\lambda)\setminus 5B} d(x,c_B)^{-n-\gamma}\wrt \mu(x) \\
	&\leq C r_B^\gamma \int_{2r_B}^{2\lambda} s^{-n-\gamma} \, s^{n-1}\wrt s
	\leq C\,.
\end{split}\]
Assume now that $a$ is a global $2$-atom, with support contained in a ball $B$ of radius $\la$.  
Then $\opK_t a$ is supported in $5B$, and arguing as in \eqref{eq:local_part}
 concludes the proof.
\end{proof}

An immediate consequence of Proposition \ref{p: limmaxop} is the following boundedness result.

\begin{corollary}
For each $K$ in $\taiV$ the maximal operator $\opK_*$ is of weak type (1,1) and bounded from $\ghu(M)$ to $L^1(M)$. 
\end{corollary}
\begin{proof}
By Definition \ref{def: AI}, we can write $K = K_1 + K_2$, where $K_1 \in \aiV$, while $K_2$ satisfies \eqref{eq:ai_l1bound}. Since the latter bound implies the $L^1$-boundedness of the maximal operator associated to $K_2$, the desired boundedness result follows by applying Proposition \ref{p: limmaxop} to $K_1$.
\end{proof}

For each $\la>0$, define the \emph{local Riesz-type potential} $\rzP_\la$ by the rule 
\[
\rzP_\la f(x)
\defeq \int_{B_\la(x)}\, \frac{|f(y)|}{d(x,y)^{n-1}} \wrt\mu(y)
\quant f \in L^1(M)\,. 
\]
From the uniform local $n$-Ahlfors condition, it readily follows that there exists a positive constant $C$, depending on $\la$, such that 
\begin{equation} \label{f: bound RieszPot}
\norm{\rzP_\la f}{1}
\leq C \norm{f}{1}
\quant f\in L^1(M)\,.
\end{equation}

\begin{lemma}\label{l: Lipschitz}
Let $L,\lambda>0$. Let $\phi: M \to \CC$ be such that
\[
\mod{\phi(x) - \phi(y)} \leq L\, d(x,y)
\]
for all $x,y \in M$ with $d(x,y) \leq \lambda$. 
Let $K : (0,1] \times M \times M \to [0,\infty)$ satisfy \ref{en:lai_supp} and \ref{en:lai_upper} of Definition~\ref{def: LAI} for some $C_1 > 0$. Then 
\[
\mod{\opK_*(\phi f)}
\leq \mod{\phi} \, \opK_*f + L\,C_1\, \rzP_\la f
\quant f \in L^1(M). 
\]
\end{lemma}

\begin{proof}
Take $f$ in $L^1(M)$, $x\in M$ and $t\in (0,1]$. 
Since $\phi$ is Lipschitz and $K$ satisfies the estimate in Definition~\ref{def: LAI}~\ref{en:lai_supp},
\[
\begin{split}
\mod{\opK_t(\phi f)(x)} 
	&  \leq  \mod{\phi(x)}\, \mod{\opK_tf(x)} + \left|\int_M K(t,x,y)[\phi(y)-\phi(x)]\, f(y)\wrt\mu(y) \right|\\
        &  \leq \mod{\phi(x)}\, \mod{\opK_tf(x)} \\
	   &\quad+  LC_1\int_{B_\la(x)}  d(x,y) \, \sup_{t\in (0,1]} t^{-n} \left(1+\frac{d(x,y)}{t}\right)^{-n-\gamma} |f(y)|\wrt \mu(y)\\
	&\leq |\phi(x)| \,\opK_*f(x) \, +LC_1 \, \rzP_\la f(x),
\end{split}
\]
because $\sup_{\al\in [0,\infty)} \al^{n}/(1+\al)^{n+\gamma} < 1$.  
By taking the supremum of both sides with respect to $t$ in $(0,1]$, we obtain the required estimate.
\end{proof}

We now show that, for any AI $K$ on $M$, the local Hardy space $\ghu(M)$ can be characterised as the space of all $f$ in $L^1(M)$ such that $\opK_*f$
is in $L^1(M)$.  Our proof hinges on the fact that a similar characterisation is already known
in the Euclidean case \cite{G}.  The main idea is to show that if $\phi$ is a suitable cutoff 
function in a neighbourhood of a point $p$ in $M$, then the composition of $\phi f$ with an 
appropriate harmonic co-ordinate map
$\eta_p^{-1}$ belongs to $\ghu(\RR^n)$.  The latter property is verified by showing that the maximal function 
of $(\phi f)\circ \eta_p^{-1}$ with respect to an appropriate AI on $\RR^n$ is in $L^1(\RR^n)$.  Part of the localisation argument is modelled over the proof of \cite[Proposition~2.2]{T}, where however a grand maximal function is considered instead of an arbitrary AI, and much more restrictive conditions on $M$ are assumed.

\bigskip

Let $Q>1$ be fixed. Due to our assumptions on $M$, as discussed in Section \ref{s: Background material}, we can find $R_0 \in (0,1]$ such that, for each $p \in M$, there exists smooth coordinates $\eta_p : B_{R_0}(p) \to \RR^n$ centred at $p$ such that
\begin{equation}\label{f: control of dist}
|X-Y|/Q \leq d(x,y) \leq Q|X-Y|
\end{equation}
for all $x,y \in B_{R_0}(p)$, where $X=\eta_p(x)$ and $Y=\eta_p(y)$, and moreover the Riemannian volume element $\mu_p$ in the coordinates $\eta_p$ satisfies
\begin{equation}\label{f: control of vol}
Q^{-n} \leq \mu_p \leq Q^n
\end{equation}
pointwise. 
Set $U_p \defeq \eta_{p} (B_{R_0}(p))$.  A direct consequence of \eqref{f: control of dist}
is that 
\[
B_{R_0/Q}^{\RR^n} (0) \subseteq U_p \subseteq B_{R_0 Q}^{\RR^n} (0).
\] 

Let $\bBall \defeq B_{R_0/2Q}^{\RR^n} (0)$.
Let $\chi : \RR^n \to [0,1]$ be a smooth cutoff function on $\RR^n$ 
supported in $\bBall$ that is equal to $1$ on $(1/2)\bBall$, and set $\Xi (X,Y) = \chi(X)\, \chi(Y)$.

Let $S$ be the $R_0/Q$-LAI on $\RR^n$ defined (similarly to \eqref{eq:LAI_example}) by
\[
S(t,X,Y) = t^{-n} \zeta((X-Y)/t),
\]
where $\zeta \in C^\infty_c(\RR^n)$ is nonnegative, supported in $B^{\RR^n}_{R_0/Q}(0)$ and satisfying
\begin{equation}\label{eq:RnLAIcond}
0 \leq \zeta(X) \leq 1, \qquad |\zeta(X)-\zeta(Y)| \leq |X-Y|^\gamma, \qquad \zeta(0)> 0.
\end{equation}

 Given a kernel $K : (0,1] \times M \times M \to [0,\infty)$, for all $p \in M$ we define $K_p^\sharp: (0,1]\times \RR^n\times \RR^n\to [0,\infty)$ by the rule 
\begin{equation}\label{f: Kpsharp}
K_p^\sharp(t,X,Y)
\defeq \Xi(X,Y)\, K(t,\eta_{p}^{-1}(X),\eta_{p}^{-1}(Y)) + \big[1-\Xi(X,Y)\big]\, S(t,X,Y).  
\end{equation}
Notice that $K_p^\sharp = S$ on $(0,1] \times (\bBall\times \bBall\big)^c$.

\begin{lemma} \label{l: modified Euclidean AI}
Suppose that $K$ is a LAI on $M$. Then, for all $p \in M$,
the function $K_p^\sharp$ defined in \eqref{f: Kpsharp} is a $R_0/Q$-LAI on $\RR^n$, with constants $C_1,C_2,C_3$ in Definition \ref{def: LAI} independent of $p \in M$.
\end{lemma}

\begin{proof}
For the sake of notational simplicity, in this proof we write $x$, $y$ and $z$ in place of $\eta_p^{-1}(X)$, $\eta_p^{-1}(Y)$, $\eta_p^{-1}(Z)$ whenever $X$, $Y$ and $Z$ are in $U_p$. 

Let us first check that $K_p^\sharp$ satisfies conditions \ref{en:lai_supp}, \ref{en:lai_upper}, \ref{en:lai_nondeg} of Definition \ref{def: LAI} for all $(t,X,Y) \in (0,1] \times \RR^n \times \RR^n$. If $(X,Y) \in (\bBall \times \bBall)^c$, then $K_p^\sharp(t,X,Y) = S(t,X,Y)$ and we are done because $S$ is a $R_0/Q$-LAI. If not, then $X,Y \in U_p$ and we can use that
\[
0 \leq K_p^\sharp(t,X,Y)
\leq \chr_{\bBall}(X) \chr_{\bBall}(Y) K(t,x,y) + S(t,X,Y)
\]
together with \eqref{f: control of dist} to deduce conditions \ref{en:lai_supp} and \ref{en:lai_upper} for $K_p^\sharp$ from the corresponding conditions for $K$ and $S$ and the fact that $|X-Y| < R_0/Q$ if both $X,Y \in \bBall$; similarly, since both $K$ and $S$ satisfy \ref{en:lai_nondeg}, we deduce that
\[
K_p^\sharp(t,X,X) = \Xi(X,Y) K(t,x,x) + (1-\Xi(X,X)) S(t,X,X) \geq C_2 t^{-n}
\]
where $C_2>0$  is the minimum of the corresponding constants for $K$ and $S$.

As for condition \ref{en:lai_lip}, by Lemma \ref{lem:lai_lip_range} it is enough to check it for all $t \in (0,1]$ and $X,Y,Z \in \RR^n$ such that
\begin{equation}\label{eq:lift_range}
\frac{4 Q^2}{R_0} |Y-Z| \leq t+\min\{|X-Y|,|X-Z|\}.
\end{equation}
If both $(X,Y),(X,Z) \notin \bBall \times \bBall$, then again we are done, since $S$ satisfies condition \ref{en:lai_lip}. If not, then $X$ and at least one of $Y,Z$ belong to $\bBall$, but then \eqref{eq:lift_range} implies that
\[
|Y-Z| \leq \frac{R_0}{4Q^2} \left(1 + \frac{R_0}{Q} \right) \leq \frac{R_0}{2 Q^2} 
\]
and consequently all of $X,Y,Z \in 2\bBall \subseteq U_p$. Hence we can write
\[\begin{split}
|K_p^\sharp(t,X,Y)-K_p^\sharp(t,X,Z)| &\leq |\Xi(X,Y) K(t,x,y)- \Xi(X,Z) K(t,x,z)| \\
&\quad + |\Xi(X,Y) S(t,X,Z)- \Xi(X,Z) S(t,X,Z)| \\
&\quad + |S(t,X,Z)- S(t,X,Z)|
\end{split}\]
and consider the three summands in the right-hand side separately. Since $S$ is a LAI on $\RR^n$, by Lemma \ref{lem:lai_cutoff} both $S$ and $(t,X,Y) \mapsto \Xi(X,Y) S(t,X,Y)$ satisfy condition \ref{en:lai_lip}, and the last two summands are dealt with. As for the first summand, in light of \eqref{f: control of dist}, the condition \eqref{eq:lift_range} implies that
\[
\frac{4}{R_0} d(y,z) \leq t+\min\{d(x,y),d(x,z)\},
\]
and we are reduced to checking that the function $(t,x,y) \mapsto \Xi(\eta_p(x),\eta_p(y)) K(t,x,y)$ satisfies condition \ref{en:lai_lip} on $M$. This however follows again from Lemma \ref{lem:lai_cutoff}, since $K$ is a LAI on $M$ and $x \mapsto \chi(\eta_p(x))$ is bounded and Lipschitz on $M$ (uniformly in $p \in M$) by \eqref{f: control of dist}.
\end{proof}

\begin{lemma}\label{lem:discretisation}
Let $\kappa > 0$. Then there exists a collection $\Discr$ of points of $M$ with the following properties.
\begin{enumerate}[label=(\roman*)]
\item\label{en:discr_cutoff}
$\{B_{\kappa}(p)\}_{p \in \Discr}$ is a locally finite cover of $M$; moreover,
for all $p \in \Discr$, there exist Lipschitz functions $\phi_p : M \to [0,1]$, with Lipschitz constant independent of $p$, 
supported in $B_{\kappa}(p)$, such that $\sum_{p\in \Discr}\phi_p=1$.

\item\label{en:discr_cutoff2}
For all $p \in \Discr$, there exist Lipschitz functions $\wt \phi_p : M \to [0,1]$, with Lipschitz constant independent of $p$, 
supported in $B_{2\kappa}(p)$, such that $\wt \phi_p=1$ on $B_{\kappa}(p)$.

\item\label{en:discr_partition}
$\Discr$ can be partitioned into finitely many subsets $\Discr_1,\dots,\Discr_N$ such that,
if $p$ and $p'$ are distinct points in $\Discr_{\ell}$ for some $\ell$ in $\{1,\ldots,N\}$, then $d(p,p') \geq 4\kappa$.
\end{enumerate}
\end{lemma}
\begin{proof}
Take any set $\Discr$ which is maximal with respect to the property 
\[
\min \big\{d(p,p') \tc p,p' \in \Discr, \, p \neq p' \big\} \geq \kappa/3.
\]
Then, by maximality,
\begin{equation}\label{eq:net}
\inf_{p \in \Discr} d(x,p) < \kappa/3   \quant x \in M.
\end{equation}
Moreover, a straightforward consequence of the uniform local Ahlfors condition (see, for instance, \cite[Lemma 1]{MVo}) is that 
\begin{equation}\label{eq:fin_card}
\sup_{x \in M} \Card(\Discr \cap B_R(x)) < \infty   \quant R>0.
\end{equation}

From \eqref{eq:net} and \eqref{eq:fin_card} it follows that $\{B_{\kappa}(p)\}_{p \in \Discr}$ is a locally finite cover of $M$.
Fix a smooth function $\psi:[0,\infty) \to [0,1]$ 
such that $\psi=1$ on $[0,\kappa/2]$ and $\psi=0$ on $(3\kappa/4,\infty)$. 
Then for every $p\in \Discr$ define $\psi_p=\psi\big(d(\cdot,p)\big)$
and $\phi_p= \psi_p/\sum_{q\in \Discr}\psi_{q}$. From \eqref{eq:net} it is clear that the locally finite sum $\sum_{q\in \Discr}\psi_{q}$ is bounded above and below by positive constants and is Lipschitz. This readily implies that the $\phi_p$ have the desired properties, and part \ref{en:discr_cutoff} is proved.

As for part \ref{en:discr_cutoff2},
take a smooth function $\wt\psi : [0,\infty) \to [0,1]$
such that $\wt\psi=1$ on $[0,\kappa]$ and $\wt\psi=0$ on $(3\kappa/2,\infty)$. Then for every $p \in \Discr$ define $\wt\phi_p = \wt\psi(d(\cdot,p))$.

Finally, for part \ref{en:discr_partition},
take as $\Discr_1$ any subset $\Discr'$ of $\Discr$ which is maximal with respect to the property
\begin{equation}\label{eq:discr_dist2}
\min \big\{d(p,p') \tc p,p' \in \Discr', \, p \neq p' \big\} \geq 4\kappa.
\end{equation}
Recursively, define $\Discr_{k+1}$ as any subset $\Discr'$ of $\Discr \setminus (\Discr_1 \cup \dots \cup \Discr_k)$ which is maximal with respect to \eqref{eq:discr_dist2}. In order to conclude, we need to show that this procedure terminates after finitely many steps, that is, $\Discr_k = \emptyset$ for some integer $k \geq 1$.

Indeed, if $\Discr_k \neq \emptyset$, then, given any $p \in \Discr_k$, by maximality of $\Discr_1,\dots,\Discr_{k-1}$, the ball $B_{4\kappa}(p)$ intersects each of $\Discr_1,\dots,\Discr_{k-1}$, and therefore $B_{4\kappa}(p)$ contains $k$ distinct points of $\Discr$ (including $p$).
However by \eqref{eq:fin_card}
the cardinality of the intersection $\Discr \cap B_{4\kappa}(x)$ is bounded by a constant independent of $x \in M$. Hence $\Discr_k \neq \emptyset$ only for finitely many integers $k$.
\end{proof}

\begin{theorem}\label{t: carmax}
Suppose that $K$ is an AI on $M$. 
If $f$ is in $L^1(M)$ and $\opK_*f$ is in $L^1(M)$, then $f$ is in $\ghu(M)$ and 
\[
	\norm{f}{\ghu}
	\leq C \left[ \norm{f}{1}+\norm{\opK_*f}{1} \right] \,,
\]
where the constant $C$ only depends on $M$ and $K$.
\end{theorem}

\begin{proof}
Let $\kappa = R_0/(8Q^2)$. In view of Lemma \ref{lem:ai_decomp} and the fact that the bound \eqref{eq:ai_l1bound} implies boundedness on $L^1(M)$ of the maximal operator associated to $K_2$, it is enough to prove the result when $K$ is a $\kappa$-LAI.

Let the collection of points $\Discr$, the sets $\Discr_1,\dots,\Discr_N$, and the families of Lipschitz cutoffs $\{\phi_p\}_{p \in \Discr}$ and $\{\wt\phi_p\}_{p \in \Discr}$ be defined as in Lemma \ref{lem:discretisation} corresponding to the above choice of $\kappa$. Recall that $\supp \phi_p \subseteq B_\kappa(p)$, and moreover $d(p,p') \geq 4\kappa$ for all $p,p' \in \Discr_\ell$ and $\ell \in \{1,\dots,N\}$. For every $\ell \in \{1,\dots,N\}$, define $\psi_\ell = \sum_{p \in \Discr_\ell} \phi_p$; note that the supports of the summands are pairwise disjoint, and in particular each $\psi_\ell$ is bounded and Lipschitz.

Let $f \in L^1(M)$.
Then $f=\sum_{\ell=1}^N (\psi_\ell f)$.  Since $\psi_\ell$ is a Lipschitz function, by Lemma~\ref{l: Lipschitz},
\[
\opK_*(\psi_\ell f)
\leq \psi_\ell \, \opK_* f + C\, \rzP_{\kappa} f.  
\]
Since $f$ is in $L^1(M)$, so is $\rzP_{\kappa} f$, by \eqref{f: bound RieszPot}.  Furthermore $\opK_*f$ is in $L^1(M)$ by assumption, whence so is $\opK_*(\psi_\ell f)$.  

Since $K$ is a $\kappa$-LAI and $\supp (\phi_p f) \subseteq B_{\kappa}(p)$ for all $p \in \Discr$, we deduce that $\supp \opK_*(\phi_p f) \subseteq B_{2\kappa}(p)$.
Consequently, for every $\ell \in \{1,\dots,N\}$, the supports of the functions $\opK_*(\phi_p f)$, $p\in \Discr_{\ell}$, are mutually disjoint.  
Therefore
\[
\sum_{p\in \Discr_\ell} \norm{\opK_*(\phi_p f)}{1} 
= \norm{\opK_* (\psi_\ell f)}{1}
\leq \norm{\opK_* f}{1} + C \norm{ f}{1},
\]
and
\begin{equation} \label{f: est cKstar phij f}
\sum_{p\in \Discr} \norm{\opK_*(\phi_p f)}{1} = \sum_{\ell=1}^N \sum_{p\in \Discr_\ell} \norm{\opK_*(\phi_p f)}{1} 
\leq C \left[ \norm{\opK_* f}{1} + \norm{ f}{1} \right],
\end{equation} 
In particular, $\opK_*(\phi_p f)$ is in $L^1(M)$ for each $p$ in $\Discr$. Similarly we see that
\begin{equation} \label{f: est phij f}
\sum_{p \in \Discr} \|\phi_p f\|_1 = \sum_{\ell = 1}^N \|\psi_\ell f \|_1 \leq C \, \| f\|_1.
\end{equation}

Recall that $\mu_p$ is the Riemannian volume element in the coordinates $\eta_p$, and set $g_p \defeq \mu_p\, \big[(\phi_p f) \circ \eta_{p}^{-1}\big]$.  Notice that the support of $g_p$
is contained in $\eta_{p}\big(B_{\kappa}(p)\big)$, which by \eqref{f: control of dist} is contained in 
$B_{\kappa Q}^{\RR^n}(0) = (1/4)\bBall$.
Consequently
\[\begin{split}
\int_{\RR^n} K^\sharp_p(t,X,Y) g_p(Y) \wrt Y &= \chi(X) \int_M K(t,\eta_p^{-1}(X),y) (\phi_p f)(y) \wrt \mu(y) \\
&\quad + \int_{\RR^n} (1-\Xi(X,Y)) \, S(t,X,Y) \, g_p(Y) \wrt Y.
\end{split}\]
Since the function $1-\Xi$ vanishes in $(1/2)\bBall\times (1/2)\bBall$, and $S$ is a $8\kappa Q$-LAI on $\RR^n$, we deduce that
\[
\sup_{t \in (0,1]}|(1-\Xi(X,Y)) \, S(t,X,Y) \, g_p(Y)| \leq C \, \chr_{[\kappa Q,8\kappa Q]}(|X-Y|) |g_p(Y)|,
\]
where $\kappa Q$ is the difference of the radii of $(1/2) \bBall$ and $(1/4) \bBall$. Therefore, if $\opK^\sharp_{p,*}$ denotes the maximal operator associated to $K^\sharp_p$, then from \eqref{f: control of vol} we deduce that
\[\begin{split}
\| \opK^\sharp_{p,*} g_p \|_{L^1(\RR^n)} &\leq C \left[ \|\opK_*(\phi_p f) \|_{L^1(M)} + \| g_p\|_{L^1(\RR^n)} \right] \\
&= C \left[ \|\opK_*(\phi_p f) \|_{L^1(M)} + \| \phi_p f\|_{L^1(M)} \right].
\end{split}\]

Recall that $K_p^\sharp$ is a LAI on $\RR^n$ by Lemma \ref{l: modified Euclidean AI}, with constants independent of $p \in \Discr$. Consequently, there exists a constant $A > 0$ (independent of $p \in \Discr$) such that $A \, K_p^\sharp$ satisfies Definition \ref{def: AI Ahlfors} on the $n$-Ahlfors space $\RR^n$, with constant $c$ independent of $p \in \Discr$.
Hence, by Theorem~\ref{thm:main},
\[
\norm{\maxG_\gamma g_p}{L^1(\RR^n)}
\leq C \norm{\opK_{p,*}^\sharp g_p}{L^1(\RR^n)}
\leq C\, \left[\norm{\opK_*(\phi_p f)}{L^1(M)} +  \norm{\phi_p f}{L^1(M)}\right],
\]
where $\maxG_\gamma$ is the grand maximal function on $\RR^n$ defined in \eqref{eq:def_grandmax}. On the other hand, if $\opS_*$ is the maximal operator associated to $S$, one clearly has the pointwise domination $\opS_* h \leq \maxG_\gamma h$ for all $h \in L^1_\loc(\RR^n)$, since the functions $S(t,X,\cdot) = t^{-n} \zeta((X-\cdot)/t)$ are in $\cutF_\gamma(X,t)$ by \eqref{eq:RnLAIcond}.
Hence it follows from \cite[Theorem 1]{G}
that $g_p$ is in $\ghu(\RR^n)$, and 
\[
\norm{g_p}{\ghu(\RR^n)}
\leq C \left[\norm{\opK_*(\phi_p f)}{1} +  \norm{\phi_p f}{1}\right].
\]
By the classical theory of $\ghu(\RR^n)$ (see, e.g., \cite[Lemma 5]{G}), the function $g_p$ admits an atomic decomposition 
$g_p =\sum_k \la_p^k a_p^k$, where $a_p^k$ are $\ghu$-atoms in $\RR^n$ at scale $\kappa Q$, and 
\begin{equation} \label{f: lajk}
\sum_k \,\bigmod{\la_p^k}
\leq C \norm{g_p}{\ghu}
\leq C  \left[\norm{\opK_*(\phi_p f)}{1}+\norm{\phi_p f}{1}\right].
\end{equation} 

Recall that $\wt \phi_p$ is a Lipschitz function on $M$ supported in $B_{2\kappa}(p)$ such that $\wt \phi_p=1$ 
on $B_{\kappa}(p)$.  Then 
\begin{equation} \label{f: phij fell}
\phi_p f
= \wt \phi_p\, \big(g_p/\mu_p\big)\circ\eta_{p}
= \wt \phi_p \, \sum_k \, \la_p^k \, \big(a_p^k/\mu_p\big)\circ \eta_{p} 
= \sum_k \, \la_p^k \, I_p^k,
\end{equation}
where $I_p^k \defeq \wt \phi_p \, \big(a_p^k/\mu_p)\circ \eta_{p}$. Notice that if $I_p^k$
is not the null function, then the supporting ball $B_p^k$ of $a_p^k$ must be contained in $B_{4\kappa Q}^{\RR^n}(0) = \bBall$. Hence, from Lemma \ref{lem:atom_ion} we deduce that the $I_p^k$ are constant multiples of $\ghu$-ions on $M$ 
at scale $\kappa Q^2$, where the constant does not depend on $p$ or $k$.

Thus, in light of \eqref{f: equivionicatomic}, from \eqref{f: phij fell} and \eqref{f: lajk} we deduce that $\phi_p f \in \ghu(M)$ and 
\[
\norm{\phi_p f}{\ghu(M)}
\leq C\, \sum_k \, \bigmod{\la_p^k}
\leq  C \, \left[ \norm{\phi_p f}{1}+ \norm{\opK_*(\phi_p f)}{1} \right].
\]
By summming over all $p \in \Discr$, and using \eqref{f: est cKstar phij f} and \eqref{f: est phij f}, we conclude that
\[
\norm{f}{\ghu(M)} 
\leq C\, \left[ \norm{\opK_*f}{1}+\norm{f}{1}\right],
\]
as required. 
\end{proof}

The following maximal characterisation of $\ghu(M)$ in terms of a \emph{single} AI
is a direct consequence of Proposition \ref{p: limmaxop} and Theorem \ref{t: carmax}. 

\begin{theorem} \label{t: max char AI}
Suppose that $K$ is an AI on $M$.  Then 
\[
	\ghu(M) = \bigl\{f\in L^1(M) \tc \opK_*f\in L^1(M)\bigr\}
\]
and there exists a positive constant $C$ such that  
\[
	C^{-1} \norm{f}{\ghu}\leq \norm{f}{1}+\norm{\opK_*f}{1}\leq C\norm{f}{\ghu} .
\]
\end{theorem}

To conclude this section, we briefly remark that the above result implies the characterisation of the Hardy space $\ghu(M)$ in terms of a $\gamma$-H\"older grand maximal function. Indeed, as in Section \ref{s:Uchiyama}, for all $x \in M$ and $r>0$ let $\cutF_\gamma(x,r)$ be the collection of all functions
$\phi : X \to \RR$ such that
\[
\supp \phi \subseteq B_r(x), \quad
|\phi(y)| \leq r^{-n}, \quad
|\phi(z) - \phi(y)| \leq r^{-n} (d(z,y)/r)^\gamma
\]
for all $y,z \in M$. Then define the $\gamma$-H\"older local grand maximal function $\maxG_\gamma$ by
\[
\maxG_\gamma f(x) = \sup_{r \in (0,1]} \sup_{\phi \in \cutF_\gamma(x,r)} \left| \int_M \phi(y) f(y) \,\wrt \mu(y) \right| 
\]
for all $f \in L^1_\loc(M)$ and $x \in M$.

\begin{corollary}\label{cor:grandmaximal}
For all $\gamma \in (0,1]$,
\[
	\ghu(M) = \bigl\{f\in L^1_\loc(M) \tc \maxG_\gamma f \in L^1(M) \bigr\}
\]
and there exists a positive constant $C$ such that  
\[
	C^{-1} \norm{f}{\ghu}\leq \norm{\maxG_\gamma f}{1}\leq C\norm{f}{\ghu}.
\]
\end{corollary}
\begin{proof}
It is immediately seen that, if $K$ is a LAI constructed as in \eqref{eq:LAI_example} with $\psi$ supported in $[0,1]$, bounded by $1$, and satisfying $|\psi(u)-\psi(v)| \leq |u-v|^\gamma$ for all $u,v \in [0,\infty)$, then $K(t,x,\cdot) \in \cutF_\gamma(x,t)$ for all $t \in (0,1]$, and therefore
\[
\opK_* f \leq \maxG_\gamma f
\]
pointwise for all $f \in L^1_\loc(M)$.
Moreover,  for all $x \in M$, by using normal coordinates centred at $x$, one easily deduces that $\lim_{t \to 0^+} \int_M K(t,x,y) \wrt\mu(y) = c$, where $c := \int_{\RR^n} \psi(|y|) \wrt y > 0$. From this it readily follows that if $f \in C_c(M)$, then $\opK_t f \to cf$ pointwise as $t \to 0^+$; a density argument then shows that, if $f \in L^1_\loc(M)$, then $\opK_t f \to cf$ in $L^1_\loc(M)$ as $t \to 0^+$, whence
\[
c |f| \leq \opK_* f
\]
pointwise a.e.\ for all $f \in L^1_\loc(M)$. This shows that, if $f \in L^1_\loc(M)$ and $\maxG_\gamma f \in L^1(M)$, then $f,\opK_* f \in L^1(M)$ as well, and by Theorem \ref{t: max char AI} we obtain that $f \in \ghu(M)$ with control of the norm.

Conversely, since $M$ is uniformly locally $n$-Ahlfors, one can easily check that
\[
\maxG_\gamma \leq C \maxM_1
\]
and that (by arguing as in the proof of Proposition \ref{p: limmaxop}) $\maxG_\gamma$ is uniformly $L^1$-bounded on $\ghu(M)$-atoms, hence $\maxG_\gamma$ is bounded from $\ghu(M)$ to $L^1(M)$ by Lemma \ref{lemma:sublinear_hardy}.
\end{proof}

\section{Characterisation of $\ghu(M)$ via heat and Poisson maximal functions}\label{s:heat_poisson}

In this section we prove that $\ghu(M)$ can be characterised in terms of the maximal operators associated 
either to the heat semigroup $\sgrH_t$ or to the Poisson semigroup $\sgrP_t$, that is,
\[
\ghuH(M) = \ghu(M) = \ghuP(M).
\]
In light of Theorem \ref{t: max char AI}, it will be enough to show that the heat and Poisson kernels on $M$ are AI in the sense of Definition \ref{def: AI}. We will actually show that a similar result holds for all semigroups $\e^{-t\opL^\alpha}$ with $\alpha \in (0,1]$.

Denote by $h_t(x,y)$ the integral kernel of the heat semigroup $\sgrH_t$.  We call the function
$(t,x,y) \mapsto h_t(x,y)$ the heat kernel.  It is well known
(see, for instance, \cite[Theorem~5.5.3 and Section 5.6.3]{SC})
that, under our assumptions on $M$,
there exist positive constants $C_1$ and $C_2$
such that 
\begin{equation} \label{f: twosided heat}
	C_1 \, \frac{\e^{-d(x,y)^2/(C_1 t)}}{t^{n/2}} 
	\leq h_t(x,y) 
	\leq C_2 \, \frac{\e^{-d(x,y)^2/(C_2 t)}}{t^{n/2}} 
\end{equation}
for every $x$ and $y$ in $M$ and $t$ in $(0,1]$.  
Furthermore (see, e.g., \cite[Theorem~6]{Da2}),
there exists positive constants $C$ and $c$ such that 
\begin{equation} \label{f: gradient heat}
\mod{\nabla_y h_t(x,y)}
\leq C \, \frac{\e^{-cd(x,y)^2/t}}{t^{(n+1)/2}} 
\end{equation}
for every $x$ and $y$ in $M$ and $t$ in $(0,1]$; in the discussion below, we will actually use the gradient bound \eqref{f: gradient heat} only for $d(x,y)$ small.

\begin{proposition} \label{p: heat}
The function
\begin{align*}
(t,x,y)\in (0,1]\times M\times M &\rightarrow h_{t^2}(x,y),
\end{align*}
is an AI of exponent $1$ in the sense of Definition~\ref{def: AI}.
\end{proposition}

\begin{proof}
Let $r = \min\{\iota_M,1\}$, where $\iota_M$ is the injectivity radius of $M$. Let $\phi : \RR \to [0,1]$ be smooth, equal to $1$ on $[-r/4,r/4]$ and supported in $[-r/2,r/2]$, and define $\Phi(x,y) = \phi(d(x,y))$. Then $\Phi(x,\cdot)$ is smooth for all $x \in M$, and $|\nabla_y \Phi(x,y)| \leq L$ for all $x,y \in M$ and some $L>0$.

Set now $K^1(t,x,y) = \Phi(x,y) h_{t^2}(x,y)$ and $K^2(t,x,y) = (1-\Phi(x,y)) h_{t^2}(x,y)$. Then clearly $K^1(t,x,y) = 0$ whenever $d(x,y) > r/2$, and moreover from \eqref{f: twosided heat} we deduce that
\[
K^1(t,x,y) \leq C t^{-n} \e^{-d(x,y)^2/(C t^2)} \leq C t^{-n} (1+d(x,y)/t)^{-n-1}
\]
and
\[
K^1(t,x,x) = h_{t^2}(x,x) \geq C t^{-n},
\]
so $K^1$ satisfies conditions \ref{en:lai_supp}, \ref{en:lai_upper} and \ref{en:lai_nondeg} of Definition \ref{def: LAI} with $\la = r/2$. In addition
\[
\nabla_y K^1(t,x,y) = h_{t^2}(x,y) \nabla_y \Phi(x,y) + \Phi(x,y) \nabla_y h_{t^2}(x,y)
\] 
and from \eqref{f: twosided heat} and \eqref{f: gradient heat} we deduce that
\[
|\nabla_y K^1(t,x,y)| \leq C t^{-n-1} \e^{-d(x,y)^2/(Ct^2)} \leq C t^{-n-1} (1+d(x,y)/t)^{-n-2}.
\]
Hence by Remark \ref{rem:LAI_diff} we deduce that $K^1$ also satisfies condition \ref{en:lai_lip} of Definition \ref{def: LAI}, and consequently is a $r/2$-LAI on $M$.

Note now that $K^2(t,x,y) = 0$ whenever $d(x,y) \leq r/4$; hence, by \eqref{f: twosided heat},
\[\begin{split}
	\sup_{t\in (0,1]} \, K^2(t,x,y)  
	&\leq C_2  \chr_{[r/4,\infty)}(d(x,y)) \, \frac{\e^{-d(x,y)^2/{2C_2}}}{d(x,y)^n}  \, \sup_{\tau>0} \, \tau^{n/2} \, \e^{-\tau/{2C_2}} \\
	&\leq C \e^{-d(x,y)^2/{2C_2}}, 
\end{split}\]
and from the volume bound \eqref{eq:exp_growth} it readily follows that $K^2$ satisfies condition \eqref{eq:ai_l1bound}. This proves that $K = K^1 + K^2$ is an AI on $M$.
\end{proof}

Now we consider the semigroups $\e^{-t \opL^\alpha}$ with $\alpha\in(0,1)$, and denote by $p_t^\alpha$ their integral kernels; the Poisson semigroup $\sgrP_t$ corresponds to $\alpha=1/2$. As it is well known (see, e.g., \cite[Section IX.11]{Y}), these semigroups can be subordinated to the heat semigroup.

\begin{lemma}\label{lem:subordination}
For $\alpha \in (0,1)$, let $F_\alpha : (0,\infty) \to \CC$ be defined by contour integration as follows:
\begin{equation} \label{f: Falpha def}
F_\alpha(s) = \frac{1}{2\pi i} \int_{\sigma-i\infty}^{\sigma+i\infty} s \, \e^{sz-z^\alpha} \wrt z
\end{equation}
for any $\sigma > 0$. Then there exists constants $C_\alpha,c_\alpha>0$ such that
\begin{equation} \label{f: Falpha est}
0 \leq F_\alpha(s) \leq C_\alpha s^{-\alpha} \, \e^{-c_\alpha s^{-\alpha/(1-\alpha)}}.
\end{equation}
Moreover
\begin{equation} \label{f: laplace tr}
e^{-z^{\alpha}} = \int_0^\infty F_\alpha(s) \, \e^{-sz} \dtt s
\end{equation}
for all $z \in \CC$ with $\Re z > 0$.
\end{lemma}
\begin{proof}
The construction of the nonnegative function $F_\alpha$ satisfying \eqref{f: laplace tr} through the contour integration \eqref{f: Falpha def} is given in \cite[Section IX.11]{Y}. It only remains to prove the upper bound in \eqref{f: Falpha est}.

Let $a = \cos(\alpha\pi/2)/3$; since $\Re z^\alpha \geq 3a |z|^\alpha$ whenever $\Re z > 0$, from the representation \eqref{f: Falpha def} we obtain that, for any $\sigma \geq (a\alpha)^{1/(1-\alpha)}$,
\[\begin{split}
F_\alpha(s) 
&\leq \frac{1}{2\pi} s \e^{s\sigma} \sigma \int_{-\infty}^\infty \e^{-3a \sigma^\alpha (1+t^2)^{\alpha/2}} \wrt t \\
&\leq C  \e^{2s\sigma-2a\sigma^\alpha},
\end{split}\]
where  the constant $C$ may depend on $\alpha$ but not on $\sigma$; by choosing $\sigma = (a\alpha/s)^{1/(1-\alpha)}$, we obtain the upper estimate in \eqref{f: Falpha est} in the case $s \leq 1$. As for the remaining case, by changing the contour of integration in \eqref{f: Falpha def} to the concatenation of the two half-lines $z = r e^{-i\theta}$ ($-\infty < -r < 0$) and $z = r \e^{i\theta}$ ($0 < r < \infty$) for any $\theta \in (\pi/2,\pi)$, we obtain, as in \cite[p.~263]{Y}, the representation
\[
F_\alpha(s) = \frac{1}{\pi} \int_0^\infty s \e^{sr \cos\theta - r^\alpha \cos(\alpha\theta)} \sin\left(sr\sin\theta-r^\alpha \sin(\alpha\theta) + \theta\right) \wrt r;
\]
by passing to the limit under the integral sign, we can also take $\theta = \pi$ in the above formula and obtain
\[\begin{split}
F_\alpha(s)
&= \frac{1}{\pi} \int_0^\infty  s \e^{-sr-r^\alpha \cos(\alpha\pi)} \sin(r^\alpha \sin(\alpha\pi)) \wrt r \\
&\leq C \int_{0}^\infty s \e^{-csr} r^\alpha \wrt r \\
&\leq C s^\alpha
\end{split}\]
for $s \geq 1$.
\end{proof}

\begin{proposition} \label{p: Poisson}
Let $\alpha \in (0,1)$. The function
\begin{align*}
(t,x,y)\in (0,1]\times M\times M &\rightarrow p^\alpha_{t^{2\alpha}}(x,y),
\end{align*}
is an AI of exponent $\min\{1,2\alpha\}$ in the sense of Definition~\ref{def: AI}.
\end{proposition}
\begin{proof}
From Lemma \ref{lem:subordination}, we deduce the subordination formula
\[
p^\alpha_{t^{2\alpha}}
= \int_0^\infty F_\alpha(s/t^2) \, h_{s} {\dtt s}.
\]
It is convenient to write $p^\alpha_{t^{2\alpha}}$ as the sum of $P_t^0$ and $P_t^\infty$, where 
\[
P_t^0
= \int_0^1 F_\alpha(s/t^2) \, h_s {\dtt s}
\qquad\text{and}\qquad
P_t^\infty
= \int_1^\infty F_\alpha(s/t^2) \, h_s {\dtt s}.
\]
We observe that, by \eqref{f: Falpha est},
\begin{equation} \label{f: sup ptinfty}
	\sup_{t\in (0,1]} \, P_t^\infty
	\leq C \int_1^\infty \sup_{t\in (0,1]} (s/t^2)^{-\alpha} \, h_s \dtt s
	\leq C \int_1^\infty \frac{h_s}{s^{1+\alpha}} \wrt s,
\end{equation}
and since $\|h_s(\cdot,y)\|_1 = 1$ for all $y \in M$ and $s >0$ we deduce that
\begin{equation} \label{f: norm sup ptO}
	\sup_{y\in M} \, \left\| \sup_{t\in (0,1]} \, P_t^\infty(\cdot,y) \right\|_1
		\leq C \int_1^\infty s^{-3/2} \wrt s
		< \infty.  
\end{equation}

Let $r \in (0,1]$ and the cutoff $\Phi : M \times M \to [0,1]$ be defined as in the proof of Proposition \ref{p: heat}. We further split $P_t^0 = P_t^{0,0} + P_t^{0,\infty}$, where $P_t^{0,0} = \Phi P_t^0$.

Note that, by \eqref{f: Falpha est}, $F_\alpha$ is a bounded function.
This and the upper estimate \eqref{f: twosided heat}
show that
\begin{equation} \label{f: sup ptO}
	\sup_{t\in (0,1]} \, P_t^{0,\infty}
	\leq C \chr_{[r/4,\infty)}(d) \int_0^1 \frac{\e^{-cd^2/s}}{s^{n/2}} {\dtt s}  
	\leq C \e^{-cd^2/2}\, \int_0^1 \frac{\e^{-cr^2/(16s)}}{s^{n/2}} {\dtt s}.
\end{equation}
The last integral is finite, hence from \eqref{eq:exp_growth} we readily see that
\begin{equation} \label{f: norm sup ptinftyO}
	\sup_{y\in M} \, \left\|\sup_{t\in (0,1]} \, P_t^{0,\infty}(\cdot,y)\right\|_1
		\leq C \sup_{y \in M} \int_M \e^{-cd(x,y)^2/2} \wrt\mu(x)
		< \infty.  
\end{equation}

The estimates \eqref{f: norm sup ptO} and \eqref{f: norm sup ptinftyO} show that $(t,x,y) \mapsto P_t^\infty(x,y) + P_t^{0,\infty}(x,y)$ satisfies condition \eqref{eq:ai_l1bound} in Definition \ref{def: AI}.

To conclude, it remains to prove that $P^{0,0}_t = P^0_t \Phi$
fulfils conditions \ref{en:lai_supp}-\ref{en:lai_lip} of Definition~\ref{def: LAI}.   
Clearly condition \ref{en:lai_supp} is satisfied with $\lambda=r/2$, due to the support of $\Phi$.
Observe that, if we set $\om = \sqrt{d^2+t^2}$, then \eqref{f: Falpha est} and the upper estimate of $h_t$ in \eqref{f: twosided heat} imply that 
\begin{equation}\label{eq:arg_omega}
\begin{split}
P_t^{0,0} 
&\leq P_t^0\, \chr_{[0,r/2]}(d) \\
&\leq C t^{2\alpha} \, \int_0^1 \, s^{-n/2-\alpha} \, \e^{-c(t^2/s)^{\alpha/(1-\alpha)}} \e^{-c d^2/s} {\dtt s} \\
&\leq C t^{2\alpha} \, \int_0^1 \, s^{-n/2-\alpha} \, \e^{-c(\om^2/s)^{\alpha}} {\dtt s} \\
&\leq C t^{2\alpha}/\om^{n+2\alpha} \leq C t^{-n} (1+d/t)^{-n-2\alpha}
\end{split}
\end{equation}
for all $t \in (0,1]$,
which 
yields the required estimate in Definition~\ref{def: LAI}~\ref{en:lai_upper} with $\gamma=\min\{1,2\alpha\}$.
Moreover, since $F_{\alpha}$ is nonnegative and does not vanish identically on $(0,1]$ (see Lemma \ref{lem:subordination}), the lower bound for $h_t$ in \eqref{f: twosided heat} implies that
\[\begin{split}
P_t^{0,0}(x,x) 
&\geq C \int_0^1 \, F_\alpha(s/t^2) \, s^{-n/2} {\dtt s} \\
&= C t^{-n} \int_{0}^{t^{-2}} F_\alpha(s) \, s^{-n/2} {\dtt s} \geq C\, t^{-n},
\end{split}\]
for all $t \in (0,1]$,
as required for condition \ref{en:lai_nondeg}. Finally, by the Leibniz rule,
\[
|\nabla_y P_t^{0,0}| \leq C P_t^0 \chr_{[0,r/2]}(d) + | \Phi \nabla_y P_t^0|.
\]
Since
\[
\nabla_y P_t^0
= \int_0^1 F_\alpha(s/t^2) \, \nabla_y h_s {\dtt s},
\]
from \eqref{f: gradient heat} we deduce, arguing as in \eqref{eq:arg_omega}, that
\[
|\Phi \nabla_y P_t^0|
\leq C t^{2\alpha} \, \int_0^1 \, s^{-(n+1)/2-\alpha} \, \e^{-c(\om^2/s)^\alpha} {\dtt s} 
\leq C t^{2\alpha}/\omega^{n+1+2\alpha}.
\]
Since $\omega \leq \sqrt{r^2/4+1}$ on the support of $\chr_{[0,r/2]}(d)$, by combining this estimate with the one for $P_t^0 \chr_{[0,r/2]}(d)$ in \eqref{eq:arg_omega}, we finally obtain that
\[
|\nabla_y P_t^{0,0}|
\leq C t^{2\alpha}/\omega^{n+1+2\alpha} 
\leq C t^{-n-1} (1+d/t)^{-n-1-2\alpha},
\]
and by Remark \ref{rem:LAI_diff} we deduce condition \ref{en:lai_lip} with $\gamma = \min\{1,2\alpha\}$.
\end{proof}

As a consequence of Propositions \ref{p: heat} and \ref{p: Poisson} and Theorem \ref{t: max char AI}, we deduce the following characterisation of $\ghu(M)$.

\begin{corollary}
For all $\alpha \in (0,1]$,
\[
\ghu(M) = \left\{ f \in L^1(M) \tc \sup_{0 < t \leq 1} |\e^{-t\opL^\alpha} f| \in L^1(M) \right\}
\]
and there exists a positive constant $C$ such that
\[
C^{-1} \|f\|_{\ghu} \leq \|f\|_1 + \left\|\sup_{0 < t \leq 1} |\e^{-t\opL^\alpha} f|\right\|_1 \leq C \|f\|_{\ghu}.
\]
In particular,
\[
\ghu(M) = \ghuH(M) = \ghuP(M)
\]
with equivalent norms.
\end{corollary}

\end{document}